\documentclass[reqno]{amsart}
\usepackage{amsmath,amsfonts,amssymb,amsthm,dsfont,bm}
\usepackage{enumitem}
\usepackage[usenames,dvipsnames]{xcolor}
\usepackage[colorlinks=true]{hyperref} 
\usepackage{verbatim}
\usepackage{graphicx}
\usepackage{mathabx}
\hypersetup{linkcolor=Violet,citecolor=PineGreen}
\newtheorem{thm}{Theorem}[section]
\newtheorem{lem}[thm]{Lemma}
\newtheorem{prop}[thm]{Proposition}
\newtheorem{cor}[thm]{Corollary}
\theoremstyle{definition}
\newtheorem{defn}[thm]{Definition}
\newtheorem{rmk}[thm]{Remark}
\numberwithin{equation}{section}
\newcommand{\N}{\mathbb{N}}
\newcommand{\R}{\mathbb{R}}
\newcommand{\Q}{\mathbb{Q}}
\newcommand{\K}{\mathcal{K}}
\newcommand{\T}{\mathcal{T}}
\newcommand{\U}{\mathcal{U}}
\newcommand{\ep}{\varepsilon}
\newcommand{\C}{\mathfrak{C}}
\newcommand{\LC}{\mathfrak{LC}}
\newcommand{\SC}{\mathfrak{SC}}
\newcommand{\nbd}{\nobreakdash}
\newcommand{\nin}{{n\in\N}}
\newcommand{\nti}{{n\to\infty}}
\newcommand*{\bigchi}{\mbox{\LARGE $\chi$}}
\DeclareMathOperator{\cls}{cls}
\DeclareMathOperator{\Int}{Int}
\makeatletter
\@namedef{subjclassname@2020}{%
  \textup{2020} Mathematics Subject Classification}
\DeclareRobustCommand{\subtitle}[1]{\\[2ex]\normalfont \emph{#1}}
\makeatother
\begin{document}
\title[Monotone skew-product semiflows for Carath\'{e}odory DE]
{Monotone skew-product semiflows for Carath\'{e}odory differential equations\\ and applications\subtitle{D\MakeLowercase{edicated to the memory of} G\MakeLowercase{en\`{e}vieve} R\MakeLowercase{augel}}}
\author[I.P.~Longo]{Iacopo P. Longo}
\author[S.~Novo]{Sylvia Novo}
\author[R.~Obaya]{Rafael Obaya}
\email[Iacopo P. Longo]{longoi@ma.tum.de}
\email[Sylvia Novo]{sylnov@wmatem.eis.uva.es}
\email[Rafael Obaya]{rafoba@wmatem.eis.uva.es}
\address[Iacopo P. Longo]{Technische Universit\"at M\"unchen,
Forschungseinheit Dynamics,
Zentrum Mathematik, M8,
Boltzmannstra{\ss}e 3,
85748 Garching bei M\"unchen, Germany.}
\address[S. Novo and R. Obaya]{
Departamento de Matem\'{a}tica Aplicada,
Universidad de Valladolid, Paseo del Cauce 59, 47011 Valladolid, Spain.}
\thanks{All authors were partly supported by MICIIN/FEDER project
RTI2018-096523-B-100 and by the EU Marie-Sk\l odowska-Curie ITN  grant H2020-MSCA-ITN-2014 643073. I.P. Longo was also supported by the European Union’s Horizon 2020 research and innovation programme under the Marie Skłodowska-Curie grant agreement No 754462.}
\subjclass[2020]{34K05, 34K25, 34A34, 37B55, 34A12}
\date{}
\begin{abstract}
The first part of the paper is devoted to  studying  the continuous dependence of the solutions of Carath\'eodory constant delay differential equations where the vector fields satisfy classical cooperative conditions.
 As a consequence,  when the set of considered vector fields is invariant with respect to the time-translation map, the continuity of the respective induced skew-product semiflows is obtained. These results are important for the study of the long-term behavior of the trajectories. In particular, the construction of semicontinuous semiequilibria and equilibria is extended to the context of ordinary and delay Carath\'eodory differential equations. Under appropriate assumptions of sublinearity,  the existence of a unique continuous equilibrium, whose graph coincides with the pullback attractor for the evolution processes, is shown.  The conditions under which such a solution is the forward attractor of the considered problem  are outlined.  Two examples of application of the developed tools are also provided.
\end{abstract}
\keywords{Carath\'eodory functions, non-autonomous Carath\'eodory delay and ordinary differential equations, continuous dependence on initial data, monotone semiflows.}
\maketitle
\section{Introduction}\label{secintro}
This paper deals with the study  of families of delay differential equations of the form $x'(t)=f\big(t,x(t),x(t-1)\big)$, $f\in E$, where  $E$ is a set of Lipschitz Carath\'{e}odory functions $f:\R \times \R^{2N}\rightarrow \R^N,(t,x,y) \rightarrow f(t,x,y)$ satisfying the usual cooperative conditions in their arguments.
The set of initial data is $C([-1,0],\R^N)$  endowed with the topology of the uniform convergence, which is a partially ordered Banach space with positive cone $C([-1,0], (\R^+)^N)$. For simplicity, these sets will be denoted by $\mathcal{C}$  and $\mathcal{C}^+$ respectively. Adequate topologies of continuity for $E$, that is, assuring the continuity of the solutions with respect to the functions in $E$ and the initial data, will be provided.
\par\smallskip
The study of the topologies of continuity for Carath\'{e}odory ordinary differential equation is a classical question with important implications in the topic of non-autonomous differential equations and its applications. In particular, Artstein~\cite{paper:ZA1,paper:ZA2,paper:ZA3}, Heunis~\cite{paper:AJH}, Miller and Sell~\cite{book:RMGS, paper:RMGS1}, Neustadt~\cite{paper:LWN}, Opial~\cite{paper:O}, Sell~\cite{book:GS} introduced and studied strong and weak topologies of integral type on the space of Lipschitz Carath\'{e}odory functions. The convergence of a sequence $(f_n)_\nin$ with respect to these topologies  requires the convergence of the integrals of the evaluation of the functions $f_n$, either pointwise in $\R^N$ (topologies $\T_P$ and $\sigma_P$), or uniformly on any bounded set of continuous functions (topology $\T_B$).
Recently, Longo et al.~\cite{paper:LNO1,paper:LNO2} completed some parts of this theory by introducing the strong and weak topologies $T_\Theta$ and $\sigma_\Theta$, where $\Theta$ is a suitable set of moduli of continuity.  When the set of the $m$-bounds of the functions in $E$ is equicontinuous, each set of uniformly bounded solutions on a compact interval admit a common modulus of continuity, and then it is possible to construct a countable set of moduli of continuity $\Theta$ such that $T_\Theta$ and $\sigma_\Theta$ become topologies of continuity for the flow map.
\par\smallskip
The extension of the previous idea to Carath\'{e}odory delay differential equations is problematic because the initial data do not generally share the same modulus of continuity of the solutions, so that the choice of a suitable set of moduli of continuity $\Theta$ is not possible.  Longo et al.~\cite{paper:LNO3} deal with this  important question and introduce the hybrid topologies  $\T_{\Theta \smash{\widehat\Theta}}$, $\sigma_{\Theta \smash{\widehat\Theta}}$,  $\T_{\Theta D}$ and $\sigma_{\Theta D}$, where $\Theta,\widehat{\Theta}$ are suitable sets of moduli of continuity and $D$ is a countable dense subset of $\R^N$. The term hybrid refers to the fact that these topologies are derived from the previous ones but treat the $x$ and $y$ components differently.\par\smallskip
 In this paper, the set $E$ will be endowed with one of these strong or weak hybrid topologies. However, in contrast with~\cite{paper:LNO3}, no assumptions on the $l$-bounds of $E$ will be considered. Nevertheless, we show that such topologies are still of continuity under very general cooperative conditions. In addition, if  $E$ is closed and invariant under the flow defined by the time-translation in the space of Carath\'{e}odory functions, the continuity of the induced local or global monotone skew-product semiflow in $E\times \mathcal{C}$ is deduced. As a consequence, we  develop topological methods to analyze the long-term behavior of the trajectories. More precisely, tools for random and deterministic monotone dynamical systems are extended to this new situation. In particular, we define and construct semicontinuous semiequilibria providing a natural version, valid in this context, of some result of Chuesov~\cite{book:chue}, Novo et al.~\cite{paper:nno2005} and Zhao~\cite{paper:Z}.
\par\smallskip
The structure and main results of the paper are organized as follows. In Section~\ref{sec:prelim} we introduce the topological spaces which are used throughout the paper and their properties. We also include a subsection on Carath\'eodory differential equations with constant delay,  stating a result of~\cite{paper:LNO3} concerning the continuous dependence of the solutions for some weak and strong hybrids topologies when the $L^1_{loc}$-equicontinuity of the $m$-bounds holds. \par\smallskip
Section~\ref{sec:monotone} focuses on the monotone case, where less conditions are needed to obtain topologies of continuity, and hence for the continuity of the induced skew-product semiflow. In particular, when the $L^1_{loc}$-equicontinuity of the $m$-bounds and a monotonicity assumption~\ref{Ky} on the $y$ component are assumed, all the strong (resp. weak) hybrid topologies stated above coincide for any $\hat\Theta$ and any $D$, provided that $\Theta$ is the  set of moduli of continuity defined by the $m$-bounds.  In~\cite{paper:LNO3}, the boundedness of the $l$-bounds was assumed instead of the monotonicity to obtain such a result.
Moreover, when the monotone condition~\ref{Kxy} with respect to both $x$ and $y$ holds, then all the previous strong topologies coincide with $\T_P$ and the weak ones with $\sigma_P$ for any countable dense set $P \subset \R^{2N}$, without any assumption on the $m$-bounds or the $l$-bounds, whatsoever. The same result holds when we assume the usual Kamke conditions~\ref{Kx} and~\ref{Ky}, as well as a very mild hypothesis~\ref{Lx} on the Lipschitz character of the components of the vector field. These are the main conclusions of this section, from which the mentioned continuity results are then derived.
\par\smallskip
Section~\ref{sec:sublinear} provides dynamical methods to study the long-term behavior of the trajectories for each of the global, continuous and monotone skew-product semiflows studied in the previous section. In particular,  the notions of semiequilibrium and semicontinuous semiequilibrium introduced in Chuesov~\cite{book:chue} and Novo et al.~\cite{paper:nno2005} are adapted to this situation, and under adequate conditions, we show that a semicontinuous semiequilibrium guarantees the existence of a semicontinuous equilibrium that determines a closed invariant set of the phase space.  When the functions of $E$ are sublinear, $E \times \mathcal{C}^+$ is positively invariant,  and a global, continuous, monotone and sublinear skew-product semiflow $\Phi$ is induced there. In this case, under appropriate assumptions, we prove the existence of invariant subsets $E_-, E_+\subset E$ such that the restriction of the semiflow to $E_-\times \mathcal{C}^+$ has an unique  equilibrium that is continuous, and whose graph is the pullback attractor for the evolution processes. Moreover, this equilibrium is also the forward attractor for the trajectories in $(E_-\cap E_+) \times \mathcal{C}^+$.
\par\smallskip
In section~\ref{sec:models} the conclusions of the paper are applied to non-autonomous models in mathematical biology. The first example is a model in population dynamics defined by a scalar Carath\'{e}odory delay differential equation for which the existence of a maximal and a minimal bounded equilibria is deduced. Under appropriate sublinearity assumptions, the existence of a unique continuous equilibrium whose graph coincides with the pullback attractor for the evolution processes is shown. The second example is the mathematical model of biochemical feedback in protein synthesis given by a system of Carath\'{e}odory ordinary differential equations, that has been extensively studied in the literature in other deterministic and random versions. See for example Selgrade~\cite{paper:selg}, Smith~\cite{book:smith1995}, Smith and Thieme~\cite{paper:smiththieme1}, Krause and Ranft~\cite{paper:krra}, Chuesov~\cite{book:chue} and Novo et al.~\cite{paper:nno2005}.
\section{Preliminary notions and results}\label{sec:prelim}
\subsection{Spaces and topologies}
We will denote by $\R^N$ the $N$\nbd-dimensional euclidean space with norm $|\cdot|$ and by $B_r$ the closed ball of $\R^N$ centered at the origin and with radius $r$. When $N=1$ we will simply write $\R$ and the symbol $\R^+$ will denote the set of positive real numbers. Unless otherwise noted, $p$ will denote a natural number $1\le p<\infty$. Moreover, for any interval $I\subseteq\R$ and any $W\subset\R^N$, we will use the following notation
\begin{itemize}[leftmargin=20pt,itemsep=3pt]
\item[] $\mathcal{C}(I,W)$: space of continuous functions from $I$ to $W$ endowed with the norm $\|\cdot\|_\infty$. In particular, we will denote by $\mathcal{C}:=\mathcal{C}([-1,0],\R^N)$.
\item[] $L^p(I,\R^N)$: space of measurable functions from $I$ to $\R^N$ whose norm is in the Lebesgue space $L^p(I)$.
\item[] $L^p_{loc}(\R^N)$: space of functions $x\colon \R \to \R^N$ such that for every compact interval $I\subset\R$, $x$ belongs to $L^p\big(I,\R^N\big)$. When $N=1$, we will simply write $L^p_{loc}$.
\end{itemize}
We will consider, and denote by $\C_p\big(\R^M,\R^N\big)$ (or simply $\C_p$ when $M=N$), the set of functions $f\colon\R\times\R^M\to \R^N$ satisfying
\begin{enumerate}[label=\upshape(\textbf{C\arabic*}),leftmargin=27pt,itemsep=2pt]
\item\label{C1}  $f$ is Borel measurable and
\item\label{C2} for every compact set $K\subset\R^M$ there exists a real-valued function $m^K\in L^p_{loc}$, called \emph{$m$-bound} in the following, such that for almost every $t\in\R$ one has $|f(t,x)|\le m^K(t)$ for any $x\in K$.
\end{enumerate}
Now we introduce the sets of Carath\'eodory functions which are used in the rest of the work.
\begin{defn}[Lipschitz Carath\'eodory functions]\label{def:LC}
A function $f\colon\R\times\R^{M}\to \R^{N}$ is said to be \emph{Lipschitz Carath\'eodory}, and we will write $f\in \LC_p\big(\R^M,\R^N\big)$ (or simply $f\in\LC_p$ when $M=N$), if it satisfies~\ref{C1},~\ref{C2} and
\begin{enumerate}[label=\upshape(\textbf{L}),leftmargin=20pt]
\item\label{L}  for every compact set $K\subset\R^M$ there exists a real-valued function $l^K\in L^p_{loc}$ such that for almost every $t\in\R$ one has $|f(t,x_1)-f(t,x_2)|\le l^K(t)\,|x_1-x_2|$ for any $x_1,x_2\in K$.
\end{enumerate}
In particular, for any compact set $K\subset\R^M$, we refer to \emph{the optimal $m$-bound} and \emph{the optimal $l$-bound} of $f$ as to
\begin{equation}
m^K(t)=\sup_{x\in K}|f(t,x)|\qquad \mathrm{and}\qquad l^K(t)=\sup_{\substack{x_1,x_2\in K\\ x_1\neq x_2}}\frac{|f(t,x_1)-f(t,x_2)|}{|x_1-x_2|}\, ,
\label{eqOptimalMLbound}
\end{equation}
respectively. Clearly, for any compact set $K\subset\R^M$ the suprema in \eqref{eqOptimalMLbound}  can be taken for a countable dense subset of $K$ leading to the same actual definition, which guarantees that the functions defined in \eqref{eqOptimalMLbound} are measurable.
\end{defn}
\begin{defn}[Strong Carath\'eodory functions]\label{def:SC}
A function $f\colon\R\times\R^M\to \R^N$ is said to be \emph{strong Carath\'eodory}, and we will write $f\in \SC_p\big(\R^M,\R^N\big)$ (or simply $f\in\SC_p$ when $M=N$), if it satisfies~\ref{C1},~\ref{C2} and
\begin{enumerate}[label=\upshape(\textbf{SC}),leftmargin=27pt]
\item\label{SC} for almost every $t\in\R$, the function $f(t,\cdot)$ is continuous.
\end{enumerate}
The concept of \emph{optimal $m$-bound} for a strong  Carath\'eodory function on any compact set $K\subset\R^N$, is defined exactly as in equation \eqref{eqOptimalMLbound}.
\end{defn}
\begin{rmk}
As regards Definitions~\ref{def:LC} and ~\ref{def:SC},  when $p=1$, we will omit the number $1$ from the notation. For example, we will simply write $\LC$ instead of $\LC_1$. Moreover, the functions which lay in the same set and only differ on a negligible subset of  $\R^{1+M}$ will be identified. Therefore, one automatically has that $\LC_p\big(\R^M,\R^N\big)\subset\SC_p\big(\R^M,\R^N\big)$.
\end{rmk}
\begin{defn}[$l_1$- and $l_2$-bounds]
\label{def:l1l2bounds}
Let us consider a function $f\in\SC_p(\R^{2N},\R^N)$. We say that \emph{$f$ admits $l_1$-bounds (resp. $l_2$-bounds)} if for every $j\in\N$ there exists a function $l_1^j\in L^p_{loc}$ (resp. $l_2^j\in L^p_{loc}$) such that for almost every $t\in\R$
\begin{equation*}
\begin{split}
|f(t,x_1,y)-f(t,x_2,y)|&\le  l_1^{j}(t)\,|x_1-x_2|\quad\text{for all }(x_1,y),(x_2,y)\in B_j\\[0.5ex]
\big(\text{resp. }|f(t,x,y_1)-f(t,x,y_2)|&\le  l_2^j(t)\,|y_1-y_2|\quad\text{for all }(x,y_1),(x,y_2)\in B_j\big).
\end{split}
\end{equation*}
If $f\in\SC_p(\R^{2N},\R^N)$ admits $l_1$-bounds (resp. $l_2$-bounds), for every $j\in\N$ we refer to the \emph{optimal $l_1$-bound} (resp. the \emph{optimal $l_2$-bound}) for $f$ on $B_j\subset\R^{2N}$ as to
\begin{equation*}
\begin{split}
 l_1^{j}(t)&=\sup_{\substack{(x_1,y),(x_2,y)\in B_j\\ x_1\neq x_2}}\frac{|f(t,x_1,y)-f(t,x_2,y)|}{|x_1-x_2|}\\
\bigg(\text{resp. } l_2^j(t)&=\sup_{\substack{(x,y_1),(x,y_2)\in B_j\\ y_1\neq y_2}}\frac{|f(t,x,y_1)-f(t,x,y_2)|}{|y_1-y_2|}\bigg)\, .
\end{split}
\end{equation*}
Notice also that, if $f\in\LC_p(\R^{2N},\R^N)$, then $f$ always admits both $l_1$-bounds and $l_2$-bounds, and  it is easy to prove that for all $t\in\R$ one has $l^j(t)\le l^j_1(t)+l^j_2(t)$, where by $l^j$ we denote the optimal $l$-bound for $f$ on $B_j$ as in \eqref{eqOptimalMLbound}.
\end{defn}
We endow the space $\SC\big(\R^M,\R^N\big)$ with strong and weak metric topologies which rely on the notion of suitable set of moduli of continuity as proposed in~\cite{paper:LNO1} and~\cite{paper:LNO2}. We, firstly recall such definition.
\begin{defn}[Suitable set of moduli of continuity]\label{def:ssmc}
We call  \emph{a suitable set of moduli of continuity}, any countable  set of non-decreasing continuous functions
\begin{equation*}
\Theta=\left\{\theta^I_j \in C(\R^+, \R^+)\mid j\in\N, \ I=[q_1,q_2], \ q_1,q_2\in\Q\right\}
\end{equation*}
such that $\theta^I_j(0)=0$ for every $\theta^I_j\in\Theta$, and  with the relation of partial order given~by
\begin{equation*}\label{def:modCont}
\theta^{I_1}_{j_1}\le\theta^{I_2}_{j_2}\quad \text{whenever } I_1\subseteq I_2 \text{ and } j_1\le j_2 \, .
\end{equation*}
\end{defn}
As follows, we recall some of the hybrid topologies on $\SC(\R^{N+M},\R^N)$ introduced in~\cite{paper:LNO3}. Notice that, as a rule, when inducing a topology on a subspace we will denote the induced topology with the same symbol used for the topology on the original space.
\begin{defn}[Hybrid topologies on $\SC(\R^{N+M},\R^N)$]\label{def:hybridtopologies}
Let $\Theta$ and $\smash{\widehat\Theta}$ be suitable sets of moduli of continuity as in Definition~\ref{def:ssmc}, $D$ be a countable dense subset of $\R^M$ and, for any $I=[q_1,q_2]$, $q_1,q_2\in\Q$ and $j\in\N$, let $\K_j^I$ and $\widehat\K_j^I$ be the sets of functions in $C(I,\R^N)$ whose modulus is bounded by $j$ and which admit $\theta^I_j$ and $\hat\theta_j^I$, respectively, as a moduli of continuity.
\begin{itemize}[leftmargin=10pt]
\item $\T_{\Theta D}$ (resp. $\sigma_{\Theta D}$) is the topology on $\SC_p(\R^{N+M},\R^N)$ (resp. $\SC(\R^{N+M},\R^N)$) generated by the family of seminorms
\begin{equation*}
\begin{split}
p_{I,\, y,\, j}(f)=\sup_{x\in\K_j^I}\left[\int_I\big|f\big(t,x(t),y\big)\big|^pdt \right]^{1/p}  ,\\
 \Bigg(\text{resp. }p_{I,\,y,\, j}(f)=\sup_{x\in\K_j^I}\left|\int_If\big(t,x(t),y\big)\,dt\right|\Bigg)
\end{split}
\end{equation*}
with $f\in\SC_p(\R^{N+M},\R^N)$ (resp. $f\in\SC(\R^{N+M},\R^N)$), $I=[q_1,q_2]$, $q_1,q_2\in\Q$, $y\in D$ and $j\in\N$. Both
$\left(\SC_p(\R^{N+M},\R^N),\T_{\Theta D}\right)$ and $\left(\SC(\R^{N+M},\R^N),\sigma_{\Theta D}\right)$ are locally convex metric spaces.\vspace{0.2cm}
\item $\T_{\Theta \smash{\widehat\Theta}}$ (resp. $\sigma_{\Theta \smash{\widehat\Theta}}$) is the topology on $\SC_p(\R^{N+M},\R^N)$ (resp. $\SC(\R^{N+M},\R^N)$)  generated by the family of seminorms
\begin{equation*}
\begin{split}
 p_{I,\, j}(f)=\sup_{x\in\K_j^I,\,y\in\widehat\K_j^{I-1}}\left[\int_I\big|f\big(t,x(t),y(t-1)\big)\big|^pdt \right]^{1/p}  ,\\
\Bigg(\text{resp. }  p_{I,\, j}(f)=\sup_{x\in\K_j^I,\,y\in\widehat\K_j^{I-1}}\left|\int_If\big(t,x(t),y(t-1)\big)\,dt\right| \Bigg)
\end{split}
\end{equation*}
with  $f\in\SC_p(\R^{N+M},\R^N)$ (resp. $f\in\SC(\R^{N+M},\R^N)$), $I=[q_1,q_2]$, $q_1,q_2\in\Q$ and $j\in\N$. One has that
$\left(\SC_p(\R^{N+M},\R^N),\T_{\Theta \smash{\widehat\Theta}}\right)$ and $\left(\SC(\R^{N+M},\R^N),\sigma_{\Theta \smash{\widehat\Theta}}\right)$ are locally convex metric spaces\end{itemize}
\end{defn}
\begin{rmk}\label{rmk:just-Theta}
The hybrid topologies  $\T_{\Theta \smash{\widehat\Theta}}$ and $\sigma_{\Theta \smash{\widehat\Theta}}$ can be equivalently constructed using families of seminorms for which $y\in\widehat\K_j^{I}$ instead of $\widehat\K_j^{I-1}$ (and of course $y$ is evaluated at $t$ instead of $t-1$). For further details, see~\cite[Lemma 2.12]{paper:LNO3}. We will use this fact consistently in the following.
Moreover, notice that when $M=0$, all the previous topologies reduce either to $ \T_\Theta$ or $\sigma_\Theta$, that is, the topologies generated by the families of seminorms
\begin{equation*}
p_{I,\, j}(f)=\!\!\sup_{x\in\K_j^I}\left[\int_I\big|f\big(t,x(t)\big)\big|^pdt \right]^{1/p} \ \left(\text{resp. }p_{I,\, j}(f)=\!\!\sup_{x\in\K_j^I}\left|\,\int_If\big(t,x(t)\big)\,dt\,\right|\right)
\end{equation*}
with $ f\in\SC_p\big(\R^N,\R^N\big)$ \big(resp. $ f\in\SC\big(\R^N,\R^N\big)$\big), $I=[q_1,q_2]$, $q_1,q_2\in\Q$, and~$j\in\N$.
\end{rmk}
\begin{rmk}
From this point on, we will only consider the case $p=1$. The reason is twofold. On one side, the case $p>1$ only applies to strong topologies, and the arguments employed in the proofs for $p=1$ can be extended with small modifications also when $p>1$; all the presented results still hold true. On the other hand, this choice allows to shorten the statement of most of the results because the case of strong and weak topologies can be presented simultaneously.
\end{rmk}
We state the following technical lemma for the weak hybrid topologies, which will be useful in the following. We skip the proof because it differs only  in minor details from the one of Lemma~2.13 in~\cite{paper:LNO2}.
\begin{lem}\label{lem:conv-subintHYBRID}
Let $D$ be a dense and countable subset of $\R^M$ and $\Theta,\widehat \Theta$ be any pair of suitable sets of moduli of continuity. Moreover, for each $j\in\N$ and  $I=[q_1,q_2]$, $q_1,q_2\in\Q$, let $\K_j^I$ and $\widehat\K_j^I$ be the compact sets in $C(I,B_j)$  which admit $\theta^I_j\in\Theta$ and $\hat\theta_j^I\in\widehat\Theta$, respectively, as moduli of continuity.
\begin{itemize}[leftmargin=18pt]
\item[\rm(i)] Consider $f\in\SC(\R^{N+M},\R^N)$. If $(x_n)_{\nin}$ is a sequence in $\K_j^I$  converging uniformly to some function $x\in\K_j^I$ and $(y_n)_{\nin}$ is a sequence in $\widehat\K_j^{I-1}$  converging uniformly to some function $y\in\widehat\K_j^{I-1}$.  Then
\begin{equation*}
\lim_{\nti}\int_{p_1}^{p_2} f\big(t,x_n(t),y_n(t-1)\big)\, dt=\int_{p_1}^{p_2} f\big(t,x(t),y(t-1)\big)\, dt\,,
\end{equation*}
whenever $p_1$, $p_2\in \Q$ and $q_1\le p_1 < p_2\le q_2$. \vspace{.1cm}
\item[\rm(ii)]  Let $(g_n)_\nin$ be a sequence in $\SC(\R^{N+M},\R^N)$ converging to some function $g$ in $\left(\SC(\R^N),\sigma_{\Theta\smash{\widehat\Theta}}\right)$ (resp. in $\left(\SC(\R^N),\sigma_{\Theta D}\right)$). Then,
\[
\begin{split}
\lim_{n\to\infty}\sup_{x\in\K_j^I,\ y\in\widehat\K_j^{I-1}}\left| \int_{p_1}^{p_2} \big[g_n\big(t,x(t),y(t-1)\big)- g\big(t,x(t),y(t-1)\big)\big]\,dt\right|=0&,\\
\Bigg(\text{resp. for any }y\in D,\quad \lim_{n\to\infty}\sup_{x\in\K_j^I}\left| \int_{p_1}^{p_2} \big[g_n\big(t,x(t),y\big)- g\big(t,x(t),y\big)\big]\,dt\right|=0&\Bigg)\\
\end{split}
\]
whenever $p_1$, $p_2\in \Q$ and $q_1\le p_1 < p_2\le q_2$.
\end{itemize}
\end{lem}
We conclude this subsection by recalling the notion of $L^1_{loc}$-equicontinuity and proving some results on the previously outlined topological spaces when such property holds.
\begin{defn}[$L^1_{loc}$-equicontinuity]\label{def:L1locequicont}
A set $S$ of positive functions in $L^1_{loc}$ \emph{is $L^1_{loc}$-equicontinuous} if for any $r>0$ and for any $\ep>0$ there exists a $\delta=\delta(r,\ep)>0$ such that, for any $-r\le s\le t\le r$, with $t-s<\delta$, the following inequality holds
\begin{equation*}
\sup_{m\in S}\int_{s}^t m(u)\,du<\ep\, .
\end{equation*}
\end{defn}
The following definition extends the previous notion to sets of Carath\'eodory functions through their $m$-bounds.
\begin{defn}[$L^1_{loc}$-equicontinuous $m$-bounds]\label{def:L1_loc equicont m-bounds}
We say that
\begin{itemize}[leftmargin=20pt]
\item[(i)] a set $E\subset\C(\R^M,\R^N)$ \emph{has $L^1_{loc}$-equicontinuous $m$-bounds}, if for any $j\in\N$ there exists a set  $S^j\subset L^1_{loc}$ of $m$-bounds of the functions of $E$ on $B_j$, such that $S^j$ is $L^1_{loc}$-equicontinuous;
\item[(ii)] \emph{$f\in\C(\R^M,\R^N)$ has $L^1_{loc}$-equicontinuous $m$-bounds} if the set $\{f_t\mid t\in\R\}$ admits $L^1_{loc}$-equicontinuous $m$-bounds.
\end{itemize}
\end{defn}
For the sake of completeness, we also include the following statement which covers a particular case of Proposition 2.19 in~\cite{paper:LNO3}. Such result guarantees that the property of $L^1_{loc}$-equicontinuity is satisfied also in the closure in $ \SC_p\big(\R^{2N},\R^N\big)$ with respect of  the topologies $\T_{\Theta D}$ and $\sigma_{\Theta D}$ (and thus also with respect to any topology of the type $\T_{\Theta \smash{\widehat\Theta}}$ and $\sigma_{\Theta \smash{\widehat\Theta}}$). In fact, the proof (that can be found in~\cite{paper:LNO3}) shows that such result is also qualitative in the sense that the limit functions satisfy the considered  inequalities with respect to the same constants.
\begin{prop}\label{prop:07.07-19:44p=1}
Let $E$ be a subset of $ \SC\big(\R^{M},\R^N\big)$ with  $L^1_{loc}$-equicontinuous $m$-bounds.  Then, the closure of $E$ in $(\SC(\R^{M},\R^N),\T)$, i.e. $\mathrm{cls}_{(\SC(\R^{M},\R^N),\T)}(E)$, has  $L^1_{loc}$-equicontinuous $m$-bounds, where $\T$ is  $\T_{\Theta D}$ or $\sigma_{\Theta D}$.
\end{prop}
\begin{rmk}\label{rmk:mod-cont-m-bounds}
If $E\subset\LC(\R^{2N},\R^N)$ has $L^1_{loc}$-equicontinuous $m$-bounds, then one can define a specific suitable set of moduli of continuity as follows: for any interval $I=[q_1,q_2]$, $q_1,q_2\in\Q$, define
\begin{equation}\label{eq:mod-cont-E}
\theta^I_j(s):=
 \sup_{t\in I,f\in E}\int_t^{t+s}m_f^j(u)\, du\, ,
\end{equation}
where, for any $f\in E$, the function $m_f^j\in L^1_{loc}$ denotes the optimal $m$-bounds of $f$ on $B_j$. In particular, if we deal with just one function $f\in\LC(\R^{2N},\R^N)$ with $L^1_{loc}$-equicontinuous $m$-bounds, one can consider the set $E=\{f_t\mid t\in\R\}$ of the time translations of $f$ and still define a suitable set of moduli of continuity as before. Notice, however, that in this case  the elements of $\Theta$ are independent of the interval $I$, that is,  for any $j\in\N$ and $I=[q_1,q_2]$, $q_1,q_2\in\Q$ one has that $\theta^I_j$ in \eqref{eq:mod-cont-E}, coincides in fact with
$\theta_j(s):= \sup_{t\in\R}\int_t^{t+s}m^j(u)\, du\,$,
where $m^j$ is the optimal $m$-bound for $f$ on $B_j$. In either case, we will say that $\Theta$ is the suitable set of moduli of continuity given by the $m$-bounds of, respectively, $E$ or $f$.
\end{rmk}
\subsection{Carath\'eodory delay differential equations}
For the sake of completeness and to set some notation, we include the statement of a theorem of existence, uniqueness and continuous variation of the solution for a Cauchy Problem of Carath\'eodory type with constant delay.  A proof can be derived by the one given for Carath\'eodory ordinary differential equations in~Coddington and Levinson~\cite[Theorems~1.1, ~1.2, ~2.2, ~4.2,  and~4.3]{book:CL}. We recall that $\mathcal{C}$ denotes the set $\mathcal{C}([-1,0],\R^N)$.
\begin{thm}\label{thm:13.04-16:49}
For any $f\in\LC(\R^{2N},\R^N)$ and any  $\phi\in\mathcal{C}$ there exists a maximal interval $I_{f,\phi}=[-1,b_{f,\phi})$ and a unique continuous function $x(\cdot,f,\phi)$ defined on $I_{f,\phi}$ which is the solution of the delay differential problem
\begin{equation}\label{eq:solLCEDDE}
\begin{cases}
x'(t)=f\big(t,x(t),x(t-1)\big)&\text{for }t>0,\\
x(t)=\phi(t)\qquad &\text{for }t\in[-1,0].
\end{cases}
\end{equation}
In particular, if $b_{f,\phi}<\infty$, then $|x(t,f,\phi)|\to\infty$ as $t\to b_{f,\phi}$.  Moreover, for any compact interval $I=[-1,b]\subset I_{f,\phi}$ and any $\ep>0$, there exists $\delta=\delta(I,\ep)>0$ such that for any $\psi\in \mathcal{C}$, if $\|\phi-\psi\|_{\mathcal{C}}<\delta$, then $x(\cdot,f,\psi)$ exists on $[-1,b]$, and $\|x(\cdot,f,\phi)-x(\cdot,f,\psi)\|_{\mathcal{C}([-1,b])}<\ep.$
\end{thm}
As usual, for each $t>0$ belonging to $I_{f,\phi}$ we will denote by $x_t(\cdot,f,\phi)$ the function of $\mathcal{C}$ defined as
\[
x_t(s,f,\phi)=x(s+t,f,\phi)  \quad \text{ whenever } s\in[-1,0]\,.
\]
For each $s\in\R$ the function $f_s\in\LC(\R^{2N},\R^N)$  denotes the time-translation of $f$
\begin{equation*}
f_s\colon\R\times \R^{2N}\to \R^N,\qquad(t,x)\mapsto f_s(t,x)=f(s+t,x).
\end{equation*}
A delay differential equation like \eqref{eq:solLCEDDE} induces a skew-product semiflow
\begin{align*}
\U\subset \R^+\times\{f_s\mid s\in\R\} \times\mathcal{C}  \to \{f_s\mid s\in\R\}\times\mathcal{C},\quad (t,g,\phi) \mapsto \big(g_t, x_t(\cdot,g,\phi)\big)\,.
\end{align*}
Instead of just one differential equation, one can consider a whole set of problems whose vector fields belong to a set $E\subset\LC(\R^{2N},\R^N)$. In order to construct a skew-product semiflow from $E$, one has to consider all the possible time-translations of any element in it (the same way we passed from $f$ to $\{f_s\mid s\in\R\}$). Instead of burdening the notation, we will then require $E$ to be invariant with respect to the function
\[
\R\times \LC(\R^{2N},\R^N)\to\LC(\R^{2N},\R^N),\quad (s,f)\mapsto f_s,
\]
that we will call \emph{base flow} in the following. Then, the skew-product semiflow induced by $E$ is
\begin{equation}\label{eq:03/07-12:50}
\U\subset \R^+\times E\times  \mathcal{C} \to E \times \mathcal{C},\quad
(t,f,\phi)\mapsto  \big(f_t, x_t(\cdot,f,\phi)\big).
\end{equation}
The aim of this paper is to provide conditions and optimal topologies on invariant closed subsets $E\subset\LC(\R^{2N},\R^N)$ with suitable monotonicity requisites, so that \eqref{eq:03/07-12:50} is continuous. Specifically, in this work no assumptions on the $l$-bounds of $E$ are made. \par\smallskip
As a matter of fact, our first result is to show that the function in \eqref{eq:03/07-12:50} can still be made continuous if restricted to a smaller subset of the phase-space (although in such a way it ceases to be a skew-product semiflow). Let $\Theta$ be a suitable set of moduli of continuity as in {\rm Definition~\ref{def:ssmc}}, and $\theta_0\in C(\R^+,\R^+)$ be a further modulus of continuity. Consider the set
\begin{equation}\label{eq:F0}
\mathcal{F}_0:=\left\{\phi\in \mathcal{C}\mid \text{mod}(\phi)\le \theta_0\right\},
\end{equation}
and the suitable set of moduli of continuity
\begin{equation}\label{eq:ThetaF0}
\overline \Theta=\big\{\bar\theta^I_j\in C(\R^+,\R^+)\mid \bar\theta^I_j(s)=\max\{\theta^I_j(s),\theta_0(s)\}\big\},
\end{equation}
and let $\T_{\Theta \smash{\overline\Theta}}$ and $\sigma_{\Theta \smash{\overline\Theta}}$ be the topologies constructed from $ \Theta$ and $\overline\Theta$ as in Definition~\ref{def:hybridtopologies}.
If $E$ is any subset  of $\LC(\R^{2N},\R^N)$, let us denote by $\overline E_{\T_{\Theta \smash{\overline\Theta}}}$ and by $\overline E_{\sigma_{\Theta \smash{\overline\Theta}}}$ the closure of $E$ in $(\LC(\R^{2N},\R^N),\T_{\Theta \smash{\overline\Theta}})$ and in $(\LC(\R^{2N},\R^N),\sigma_{\Theta \smash{\overline\Theta}})$, respectively, and let $\U_{\T_{\Theta \smash{\overline\Theta}}}\subset\R^+\times\overline E_{\T_{\Theta \smash{\overline\Theta}}}\times \mathcal{F}_0$ and  $\U_{\sigma_{\Theta \smash{\overline\Theta}}}\subset\R^+\times\overline E_{\sigma_{\Theta \smash{\overline\Theta}}}\times \mathcal{F}_0$ be defined by
\begin{equation*}
\U_{\T_{\Theta \smash{\widehat\Theta}}}=\!\bigcup_{\substack{f\in\overline E_{\T_{\Theta \smash{\widehat\Theta}}},\\[0.5ex]\phi\in \mathcal{F}_0}}\!\!\{(t,f,\phi)\mid t\in I_{f,\phi}\}\ \quad\text{and}\ \quad
\U_{\sigma_{\Theta \smash{\widehat\Theta}}}=\!\bigcup_{\substack{f\in \overline E_{\sigma_{\Theta \smash{\widehat\Theta}}},\\[0.5ex]\phi\in \mathcal{F}_0}}\!\!\{(t,f,\phi)\mid t\in I_{f,\phi}\}
\,.
\end{equation*}
\begin{thm}\label{thm:ContinuitySolutions_F0_to_C}
Consider $E\subset\LC(\R^{2N},\R^N)$  with $L^1_{loc}$-equicontinuous $m$-bounds and let $\Theta$ be the suitable set of moduli of continuity given by the $m$-bounds in {\rm Remark~\ref{rmk:mod-cont-m-bounds}}.
\begin{itemize}[leftmargin=20pt,itemsep=2pt]
\item[\rm (i)]   Let $\big(\phi_{n}\big)_\nin$ be a sequence converging uniformly to $\phi$ in $\mathcal{C}$ and let $\theta_0$ be the modulus of continuity shared by the functions $\{\phi_n\mid\nin\}\cup\{\phi\}$.
    Furthermore, let  $\overline\Theta$ be the suitable set of moduli of continuity constructed in \eqref{eq:ThetaF0}. If $(f_n)_\nin$ is a sequence in $E$ converging to $f$ in $(\LC(\R^{2N},\R^N),\sigma_{\smash{\Theta\overline\Theta}})$, then, with the notation of {\rm Theorem~\ref{eq:solLCEDDE}}, one has that
\begin{equation*}
 x(\cdot,f_n,\phi_{n}) \xrightarrow{\nti}  x(\cdot,f,\phi)
\end{equation*}
uniformly in any $[-1,T]\subset I_{f,\phi}$.
\item[\rm (ii)] If additionally $E$ is invariant with respect to the base flow, $\overline E$ denotes the closure of $E$ in $(\LC(\R^{2N},\R^N),\T)$, where $\T\in\{\T_{\Theta \smash{\overline\Theta}},\,\sigma_{\Theta \smash{\overline\Theta}}\}$, and $\mathcal{F}_0$ is the set constructed as in \eqref{eq:F0} then  the map
\begin{equation*}
\quad\ \Phi\colon \U_\T\subset \R^+\times\overline E\times \mathcal{F}_0  \to \overline E\times \mathcal{C},\quad
 (t,f,\phi)\mapsto  \big(f_t, x_t(\cdot,f,\phi)\big)
\end{equation*}
is continuous.
\end{itemize}
\end{thm}
\begin{proof}
The proof of (i) is given in~\cite[Theorem 3.6]{paper:LNO3}. On the other hand, (ii) is a consequence of (i) and of the continuity of the function $(t,f)\mapsto f_t$  proved in~\cite[Theorem~3.8]{paper:LNO3}.
\end{proof}
\section{Monotone semiflows}\label{sec:monotone}
In this section, we address the problem of the global continuity for the skew-product semiflows induced by Carath\'eodory delay differential equations whose vector fields satisfy a specific property of monotonicity. We initially recall the definition of Kamke's conditions. In this setup, they split into two properties of monotonicity, ~\ref{Kx} and ~\ref{Ky}, one for each spatial variable. Then, we prove several technical results. For example, we show that ~\ref{Kx} and ~\ref{Ky} are robust with respect to the limits in the hybrid topologies, and that~\ref{Ky} is sufficient to provide equivalence of the hybrid topologies (either weak or strong). As a consequence, we obtain that the  skew-product semiflow induced by a set of Carath\'eodory functions satisfying ~\ref{Ky}, is continuous if endowed with $\T_{\Theta D}$ or $\sigma_{\Theta D}$. Finally, we recall the definition of monotone skew-product semiflows and prove that ~\ref{Kx} and ~\ref{Ky} together, characterize a monotone skew-product semiflow. Hence, we obtain the continuity of such semiflows as a corollary of the previous results. The last part of the section deals with a stronger assumption of monotonicity which allows to obtain two additional results of continuity of the skew-product semiflow  when pointwise (and thus weaker) topologies are used. Moreover, we relate the previous results to Carath\'eodory ordinary differential equations.
\par\smallskip
From the usual componentwise strong partial ordering in $\R^N$, that is,
\begin{equation*}
\begin{split}
y  \leq z \quad & \Longleftrightarrow \quad y_i\leq z_i\quad
\text{ for}
\quad i=1,\ldots,N\,,\\
 y  < z  \quad   & \Longleftrightarrow  \quad  \,y \,\leq\, z\quad\text{ and }y_i < z_i
 \; \text{ for some}\; i\in\{1,\ldots,N\}\,,\\
y  \ll z \quad & \Longleftrightarrow \quad y_i < z_i\quad \text{
for}\quad i=1,\ldots,N\,.
\end{split}
\end{equation*} we obtain a
strong partial ordering on $\mathcal{C}=\mathcal{C}([-1,0],\R^N)$ defined by
\begin{align}\nonumber
\phi \leq \psi \quad & \Longleftrightarrow \quad \phi(s)\leq
\psi(s)\; \text{ for each }\;s\in[-1,0]\,,\\ \label{orderingC}
 \phi < \psi  \quad & \Longleftrightarrow  \quad \phi \leq \psi\quad\text{ and
 }\; \phi\neq \psi  \,, \\ \nonumber
\phi \ll \psi \quad & \Longleftrightarrow \quad \phi(s)\ll \psi(s)\; \text{
for each }\;s\in[-1,0]\,.
\end{align}
The positive cone is $\mathcal{C}^+:=C([-1,0], (\R^+)^N)=\{\phi\in \mathcal{C}\mid \phi\geq 0\}$  with nonempty interior $\Int \mathcal{C}^+=\{\phi\in \mathcal{C}\mid \phi\gg 0\}$.
\begin{defn}[Kamke's conditions]\label{Kamke}
We say that $ f\in\LC(\R^{2N},\R^N)$ satisfies
\begin{enumerate}[label=\upshape(\textbf{K}$_x$),leftmargin=27pt]
\item\label{Kx} if for any $a,b,c\in\R^N$ with $a\le b$ and $a_k=b_k$ for some $k\in\{1,\dots,N\}$, then
\[
f_k(t,a,c)\le f_k(t,b,c),\qquad \text{for a.e. } t\in\R;\\[1ex]
\]
\end{enumerate}
\begin{enumerate}[label=\upshape(\textbf{K}$_y$),leftmargin=27pt]
\item\label{Ky} if for any $a,b,c\in\R^N$ with $b\le c$, then
\[
f_k(t,a,b)\le f_k(t,a,c),\qquad \text{for all } k=1,\dots,N,\text{ and a.e. } t\in\R.\\[1ex]
\]
\end{enumerate}
If both ($K_x$) and ($K_y$) hold true, we say that $f$ satisfies the \emph{Kamke's conditions} or, equivalently, that $f$ is \emph{cooperative}.
\end{defn}
Our first result is to show that ~\ref{Kx} and ~\ref{Ky} can be propagated through the limits in any of the hybrid topologies considered in Definition~\ref{def:hybridtopologies}.
\begin{prop}\label{prop:Kx-Ky propagate to closure}
  Consider $E\subset\LC(\R^{2N},\R^N)$. Let $\Theta$ be any suitable set of moduli of continuity as in {\rm Definition~\ref{def:ssmc}}, $D$ a countable dense  subset  of $\R^N$, and  $\overline E$ the closure of $E$ in $(\LC(\R^{2N},\R^N),\sigma_{\Theta D})$. Then
\begin{itemize}[leftmargin=19pt,itemsep=2pt]
\item[\rm (i)] if any function in $E$ satisfies ~\ref{Ky}, then also any function in $\overline E$ does;
\item[\rm (ii)] if any function in $E$ satisfies ~\ref{Kx}, then also any function in $\overline E$ does.
\end{itemize}
\end{prop}
\begin{proof}
 (i) Consider $f\in \overline E$ and a sequence $(f_n)_\nin$ of functions in $E$ converging to $f$ with respect to $\sigma_{\Theta D}$. First we take $a\in\R^N$ and we assume that $b$ and $c$ belongs to $D$. Then, from the convergence in $\sigma_{\Theta D}$ (see Definition~\ref{def:hybridtopologies}), fixed any compact interval $I\subset \R$ with rational extrema we deduce that
\begin{equation}\label{eq:convg}
\lim_{n\to\infty}\int_I f_n(s,a,b)=\int_I f(s,a,b)\,ds \;\; \text{and} \;\; \lim_{n\to\infty}\int_I f_n(s,a,c)=\int_I f(s,a,c)\,ds\,.
\end{equation}
Moreover,  considering $b\leq c$, $k=1,\ldots,N$ and any $t$, $h\in\Q$ with $h\geq 0$, from~\ref{Ky} we obtain
\begin{equation*}
\frac{1}{h}\int_t^{t+h}\left(f_n\right)_k(s,a,b)\,ds\le \frac{1}{h}\int_t^{t+h}\left(f_n\right)_k(s,a,c)\,ds,
\end{equation*}
which thanks to \eqref{eq:convg}, as $\nti$ provides
\begin{equation*}
\frac{1}{h}\int_t^{t+h}f_k(s,a,b)\,ds\le \frac{1}{h}\int_t^{t+h}f_k(s,a,c)\,ds.
\end{equation*}
Due to the continuity of the integral operator, the previous inequality is satisfied for any $t,h\in\R$, with $h\ge 0$. Therefore, using Lebesgue's theorem as $h\to0$ we conclude that
\[
f_k(t,a,b)\le f_k(t,a,c),\qquad\text{ for a.e. } t\in\R\,,
\]
provided that $b$, $c\in D$ and $b\leq c$. Finally, from~\ref{L} and the density of $D$ we deduce the result for any $b$, $c\in \R^N$
which concludes the proof of (i). The proof of (ii) is omitted because it is similar.
\end{proof}
Next, we prove a technical lemma which allows us to pass from~\ref{Ky} and~\ref{Kx} (which are formulated pointwise) to a condition of monotonicity which involves continuous functions.
\begin{lem}\label{lem:f(t,a,b)->f(t,x(t),y(t))}
Consider $f\in\LC(\R^{2N},\R^N)$. Then
\begin{itemize}[leftmargin=19pt,itemsep=2pt]
\item[\rm (i)] if $f$ satisfies~\ref{Ky}, for any interval $I\subset\R$ and any functions $x,\,y,\,z\in \mathcal{C}(I,\R^N)$ such that $y(t)\le z(t)$ for all $t\in I$,  one has that
\begin{equation}\label{eq:18/09-19:12}
f\big(t,x(t),y(t)\big)\le f\big(t,x(t),z(t)\big),\qquad \text{for a.e. } t\in I;
\end{equation}
\item[\rm (ii)] if $f$ satisfies~\ref{Kx}, for any interval $I\subset\R$ and any functions $x,\,y,\,z\in \mathcal{C}(I,\R^N)$ such that $x(t)\le y(t)$ and $x_k(t)=y_k(t)$ for all $t\in I$ and some $k\in{1,\ldots,N}$,  one has that
\begin{equation*}
f_k\big(t,x(t),z(t)\big)\le f_k\big(t,y(t),z(t)\big),\qquad \text{for a.e. } t\in I.
\end{equation*}
\end{itemize}
\end{lem}
\begin{proof}
  Let $D=\{s_n \mid n\in\N\}$ be a dense subset of $I$. From~\ref{Ky} we know that given $n\in\N$ there is a subset $J_n\subset I$ of full
 measure such that
\begin{equation}\label{Jn}
 f\big(t,x(s_n),y(s_n)\big)\le f\big(t,x(s_n),z(s_n)\big)\qquad \text{for each } t\in J_n.
\end{equation}
Next we consider the subset $J=\bigcap_{n=1}^\infty J_n\subset I$, also of full measure and fix $t\in J$. From the density of $D$ we find a subsequence
$(s_{n_k})_{k\in N}$ such that $\lim_{k\to\infty} s_{n_k}=t$, and
from~\eqref{Jn} we deduce that
\begin{equation}\label{ineqsn}
f\big(t,x(s_{n_k}),y(s_{n_k})\big)\le f\big(t,x(s_{n_k}),z(s_{n_k})\big)\qquad \text{for each } k\in\N\,.
\end{equation}
Moreover, from~\ref{L} and the continuity of the functions $x$, $y$ and $z$ we obtain
\begin{align*}
\lim_{k\to\infty}f\big(t,x(s_{n_k}),y(s_{n_k}\big)&=f\big(t,x(t),y(t)\big)\,,\\
\lim_{k\to\infty}f\big(t,x(s_{n_k}),z(s_{n_k}\big)&=f\big(t,x(t),z(t)\big)\,,
\end{align*}
which together with~\eqref{ineqsn} yields~\eqref{eq:18/09-19:12} for $t\in J$, and finishes the proof of (i). The proof of (ii) is omitted because it is similar.
\end{proof}
The previous results allow to obtain a crucial outcome. All the considered weak (resp. strong) topologies coincide on closed subsets of $\LC(\R^{2N},\R^N)$ where all the elements satisfy the property ~\ref{Ky}.
\begin{prop}\label{prop:K_y=>equiv topol}
 Let $E\subset\LC(\R^{2N},\R^N)$, $\Theta$ and  $\widehat\Theta$ be any pair of suitable sets of moduli of continuity as in {\rm Definition~\ref{def:ssmc}}, and $D$ a countable dense subset of $\R^N$. If all the functions in $E$ satisfy ~\ref{Ky}, then
\[
\cls_{(\LC,\sigma_{\Theta\smash{\widehat\Theta}})} (E)=\cls_{(\LC,\sigma_{\Theta D})} (E)\quad\text{and}\quad \cls_{(\LC,\T_{\Theta\smash{\widehat\Theta}})} (E)=\cls_{(\LC,\T_{\Theta D})} (E).
\]
\end{prop}
\begin{proof}
Firstly, notice that for any pair $\Theta$ and  $\widehat\Theta$ of suitable sets of moduli of continuity and any countable dense subset $D\subset \R^N$, $\sigma_{\Theta\smash{\widehat\Theta}}\ge\sigma_{\Theta D}$. Then, we will complete the proof proving that the  convergence in $\sigma_{\Theta D}$ implies the one in $\sigma_{\Theta\smash{\widehat\Theta}}$.\par\smallskip
Let $(f_n)_\nin$ be a sequence of functions in $E$ converging to some $f\in\LC(\R^{2N},\R^N)$ with respect to $\sigma_{\Theta D}$. Fixed any interval $I=[p,q]$, with $p,q\in\Q$, we consider its partition in $m$ subintervals of the same length, that is, for $m\in\N$,
\[
I=\bigcup_{i=0}^{m-1} I_i,\quad\text{where } I_i=\left[p+i\,\frac{q-p}{m},p+(i+1)\frac{q-p}{m}\right],\, i=0,\dots, m-1.
\]
Moreover, for any $y\in\widehat\K^I_j$, with $j\in\N$, and any $m\in\N$, let $\underline y^m:I\to\R^N$ and $\overline y^m:I\to\R^N$ be the simple functions defined by
\begin{equation*}\label{eq:16/04-13:22}
\underline y^m_k(t)=\sum_{i=0}^{m-1} \bigchi_{I_i}(t)\inf_{s\in I_i}y_k(s)\qquad\text{and}\qquad
\overline y^m_k(t)=\sum_{i=0}^{m-1} \bigchi_{I_i}(t)\sup_{s\in I_i}y_k(s),
\end{equation*}
where $\bigchi_J$ is the characteristic function of the interval $J$. Then, clearly one has
\[
\underline y^m(t)\le y(t)\le \overline y^m(t),\quad \text{for all }t\in I\,.
\]
Fix $0<\ep<1/3$. For every $j\in\N$ there are  $r_j\in\N$ and  $z_1,\dots,z_{r_j}\in D\cap B_{j+1}$ such that $B_{j+1}\subset \bigcup_{i=1}^{r_j}B_\ep(z_i)$.
On the other hand, by construction and since $\widehat\K^I_j$ is compact,  there is an $m_0\in\N$ so that
\[
\|\underline y^{m_0}-y\|_{L^\infty(I)}<\ep\quad\text{and}\quad \|\overline y^{m_0}-y\|_{L^\infty(I)}<\ep,\qquad\text{for all } y\in \widehat\K^I_j\,.
\]
From this, for every $y\in\widehat\K^I_j$ there are simple functions $\underline y^{m_0}_D,\overline y^{m_0}_D\colon I\to\{z_1,\dots,z_{r_j}\}$ such that, denoted by $v$ the unitary vector $v=(1/\sqrt{N},\ldots, 1/\sqrt{N})$,
\[
\big|(\underline y^{m_0}(t)-\ep \, v)-\underline y^{m_0}_D(t)\big|<\ep\quad\text{and}\quad \big|(\overline y^{m_0}(t)+\ep\,v)-\overline y^{m_0}_D(t)\big|<\ep
\]
for each $t\in I$. One can easily check that
\begin{align*}
\underline y^{m_0}_D(t)&\le \underline y^{m_0}(t)\le y(t)\le \overline y^{m_0}(t)\le\overline y^{m_0}_D(t) ,\quad \text{for all }t\in I,\\
\|\underline y^{m_0}_D-y\|_{L^\infty(I)}& <3\,\ep\quad\text{and}\quad \|\overline y^{m_0}_D-y\|_{L^\infty(I)}<3 \,\ep,\quad\text{for all }y\in  \widehat\K^I_j\,.
\end{align*}
Therefore, from Lemma~\ref{lem:f(t,a,b)->f(t,x(t),y(t))}  and using the convergence in $\sigma_{\Theta D}$, we have that there is an $n_0$ such that for each $n> n_0$,  $i=0,\dots,m_0-1$ and $k=1,\dots, N$,
\begin{equation}\label{eq:16/04-13:27}
\begin{split}
&\sup_{x\in\K^I_j,\, y\in\widehat \K^I_j}\int _{I_i}\big[\left(f_n\right)_k\big(t,x(t), y(t)\big)-f_k\big(t,x(t), y(t)\big)\big]\,dt\\
&\qquad\quad\le \sup_{x\in\K^I_j,\, y\in\widehat \K^I_j}\int _{I_i}\big[\left(f_n\right)_k\big(t,x(t), \overline y^{m_0}_D(t)\big)-f_k\big(t,x(t), \overline y^{m_0}_D(t)\big)\big]\,dt \\
&\qquad\quad\qquad\qquad+\sup_{x\in\K^I_j,\, y\in\widehat \K^I_j}\int _{I_i}\big|f_k\big(t,x(t), \overline y^{m_0}_D(t)\big)-f_k\big(t,x(t), y(t)\big)\big|\,dt\\
&\qquad\quad\le \frac{3\,\ep}{m_0}+\sup_{x\in\K^I_j,\, y\in\widehat \K^I_j}\int _{I_i}l^{2(j+1)}_f(t)\,\|\overline y^{m_0}_D-y\|_{L^\infty(I)}\,dt \\
& \qquad\quad \le \frac{3\,\ep}{m_0}+3\,\ep\int _{I_i}l^{2(j+1)}_f(t)\,dt\le \frac{3\,\ep}{m_0}+3\,\ep\int _{I}l^{2(j+1)}_f(t)\,dt\,.
\end{split}
\end{equation}
On the other hand, using $\underline  y^{m_0}_D$ and with analogous reasoning, one also has that
\begin{equation*}
\sup_{x\in\K^I_j,\, y\in\widehat \K^I_j}\int _{I_i}\!\!\big[f_k\big(t,x(t), y(t)\big)-\left(f_n\right)_k\big(t,x(t), y(t)\big)\big]\,dt \le\frac{3\,\ep}{m_0}+3\,\ep \int _{I}l^{2(j+1)}_f(t)\,dt \,.
\end{equation*}
Therefore, taking into account that $f_n$ and $f$ have $N$ components and there are $m_0$ subintervals $I_i$ whose union is $I$, the two previous chains of inequalities yield to
\begin{equation*}
\sup_{x\in\K^I_j,\, y\in\widehat \K^I_j}\left|\int _{I}\big[f_n\big(t,x(t), y(t)\big)-f\big(t,x(t), y(t)\big)\big]\,dt\right|\le 3N \ep+ 3N \ep\int _{I}l^{2(j+1)}_f(t)\,dt\,,
\end{equation*}
and from  the arbitrariness on $\ep$ we obtain the convergence in $\sigma_{\Theta\smash{\widehat\Theta}}$, as claimed.
\par\smallskip
The proof for the strong topologies can be carried out similarly. Maintaining the previous notation and arguing again for each $k=1,\dots, N$, we can write
\begin{align*}
\big|\!\left(f_n\right)_k\!\big(t,x(t), y(t)\big)\!-\!f_k\big(t,x(t), y(t)\big)\big|
 &\le \left|\left(f_n\right)_k\big(t,x(t),  \overline y^{m_0}_D(t)\big)\!-\!f_k\big(t,x(t), y(t)\big)\right|\\
& \;\quad +\left|\left(f_n\right)_k\big(t,x(t), \underline  y^{m_0}_D(t)\big)\!-\!f_k\big(t,x(t), y(t)\big)\right|.
\end{align*}
Then, we obtain the conclusion by using, for each one of the integrals of the terms on the right-hand side, the same chain of inequalities \eqref{eq:16/04-13:27} employed before.
\end{proof}
The importance of the equivalence of the topologies in obtaining the continuity of the skew-product semiflow generated by a set of Carath\'eodory differential equations has been extensively treated in~\cite{paper:LNO3}. In that work the equivalence was obtained via a thorough topological study of the $l$-bounds of a set of Carath\'eodory functions. Hence, it is not surprising that a new result of continuity can be achieved.
\begin{thm}\label{thm:K_y=>cont skew-prod}
Consider $E\subset\LC(\R^{2N},\R^N)$  with $L^1_{loc}$-equicontinuous $m$-bounds and such that every element of $ E$ satisfies~\ref{Ky}. Let $\Theta$ be the suitable set of moduli of continuity given by the $m$-bounds in {\rm Remark~\ref{rmk:mod-cont-m-bounds}}, $D$ a countable dense subset of $\R^n$, and $\sigma_{\Theta D},\, \T_{\Theta D}$ the topologies of {\rm Definition~\ref{def:hybridtopologies}}.
\begin{itemize}[leftmargin=20pt,itemsep=2pt]
\item[\rm (i)]   If $\big(\phi_{n}\big)_\nin$ is a sequence converging uniformly to $\phi$  in $\mathcal{C}$, and $(f_n)_\nin$ is a sequence in $E$ converging to $f$ in $(\LC(\R^{2N},\R^N),\sigma_{\smash{\Theta D}})$, then, with the notation of {\rm Theorem~\ref{eq:solLCEDDE}}, one has that
\begin{equation*}
 x(\cdot,f_n,\phi_{n}) \xrightarrow{\nti}  x(\cdot,f,\phi)
\end{equation*}
uniformly in any $[-1,T]\subset I_{f,\phi}$.
\item[\rm (ii)]  if additionally $E$ is invariant with respect to the base flow, and $\overline E$ denotes the closure of $E$ in $(\LC(\R^{2N},\R^N),\T)$, where $\T\in\{\sigma_{\Theta D}, \T_{\Theta D}\}$ then  the map
\begin{equation*}
\begin{split}
\quad\Phi\colon \U\subset \R^+\times\overline E\times\mathcal{C}\to \overline E\times\mathcal{C},\quad (t,f,\phi)\mapsto  \big(f_t, x_t(\cdot,f,\phi)\big),
\end{split}
\end{equation*}
is a continuous skew-product semiflow.
\end{itemize}
\end{thm}
\begin{proof}
As regards (i), consider  a sequence $(f_n)_\nin$ in $E$ converging to some $f\in\overline E$ with respect to $\sigma_{\Theta D}$, a sequence $(\phi_n)_\nin$ converging  to some $\phi$ in $\mathcal{C}$ endowed with the compact-open topology, and a sequence $(t_n)_\nin$ converging to some $t$ in $\R$. From Ascoli-Arzel\'a's theorem, there exists a common modulus of continuity $\theta_0\in C(\R^+,\R^+)$ for $\{\phi_n\mid\nin\}\cup\{\phi\}$ and so one can construct the sets $\mathcal{F}_0$ and $\overline \Theta$ as in \eqref{eq:F0} and \eqref{eq:ThetaF0}, respectively, and the topology $\sigma_{\Theta \smash{\overline\Theta}}$ as in Definition~\ref{def:hybridtopologies}. Thanks to Proposition~\ref{prop:K_y=>equiv topol}, one has that  $(f_n)_\nin$ converges to $f$ also with respect to $\sigma_{\Theta \smash{\overline\Theta}}$. Therefore, one has the result using Theorem~\ref{thm:ContinuitySolutions_F0_to_C}.\par\smallskip
On the other hand, (ii) is a consequence of (i), of Propositions~\ref{prop:07.07-19:44p=1} and~\ref{prop:Kx-Ky propagate to closure},  and of the continuity of the function $(t,f)\mapsto f_t$ proved in~\cite[Theorem~3.8]{paper:LNO3}.
\end{proof}
As follows we relate the previous results to \emph{monotone skew-product semiflows}. Recall the strongly ordering defined on $\mathcal{C}$ in~\eqref{orderingC}.
\begin{defn}[Monotone skew-product semiflow]
Given a set $E\subset\LC(\R^{2N},\R^N)$ invariant with respect to the base flow, the skew-product semiflow
\begin{equation}\label{eq:03/10-14:05}
\begin{split}
\Phi\colon \U\subset \R^+\times E\times\mathcal{C}\to E\times\mathcal{C},\quad (t,f,\phi)\mapsto  \big(f_t, x_t(\cdot,f,\phi)\big)
\end{split}
\end{equation}
is said to be \emph{monotone} if
\[
x(t,f,\phi)\le x(t,f,\psi)\quad\text{whenever  }t\in I_{f,\phi}\cap  I_{f,\psi},\,f\in E \text{  and } \phi\leq \psi\,.
\]
\end{defn}
 Monotone skew-product semiflows have been extensively studied in the literature. As follows, we show that the skew-product semiflow generated by a set of Carath\'eodory delay differential equations is monotone if and only if all the vector fields are cooperative.
\begin{thm}\label{thm:monoton<=>Kamke}
Consider $E\subset\LC(\R^{2N},\R^N)$ invariant with respect to the base flow. The following statements are equivalent:
\begin{itemize}[leftmargin=27pt,itemsep=2pt]
\item[\rm (a)] Every $f\in E$ satisfies ~\ref{Kx} and ~\ref{Ky}.
\item[\rm (b)] The skew-product semiflow \eqref{eq:03/10-14:05} is monotone.
\end{itemize}
\end{thm}
\begin{proof}
(a)$\Rightarrow$(b) Consider $f\in E$, $\phi,\;\psi\in\mathcal{C}$ with $\phi\leq \psi$. Notice that $x(t, f,\phi)$ and $x(t,f,\psi)$ are the solutions of the Carath\'{e}odory ordinary differential problems
\[
\begin{cases}
x'(t)=f\big(t,x, \phi(t-1)\big),\\
x(0)=\phi(0),
\end{cases}
\quad\text{and}\qquad
\begin{cases}
 x'(t)=f\big(t,x, \psi(t-1)\big),\\
x(0)=\psi(0).
\end{cases}
\]
for $t\in[0,1]$, assuming that  $[0,1]\subset I_{f,\phi}\cap I_{f,\psi}$. Then $x(t,f,\phi)\leq x(t,f,\psi)$ for each $t\in[0,1]$ follows from Walter~\cite[Theorem 2]{paper:Walter}. This argument can be iterated on successive intervals of length smaller than one to obtain the inequality in any compact interval contained in $I_{f,\phi}\cap I_{f,\psi}$.\par\smallskip
(b)$\Rightarrow$(a) We prove~\ref{Kx}. Consider $f\in E$, and $a,b,c\in\R^N$ such that $a\le b$ and $a_k=b_k$ for some $k\in\{1,\dots,N\}$. Fixed any $t_0\in\R$,  consider $\phi$, $\psi\in\mathcal{C}([t_0-1,t_0],\R^N)$ defined by
$
\phi(s)=(t_0-s)\,c+(s-t_0+1)\,a\,$ and $\psi(s)=(t_0-s)\,c+(s-t_0+1)\,b\,$,
and let $x(t)$ and $y(t)$ be the solutions of the Carath\'{e}odory delay  differential problems
\[
\begin{cases}
x'(t)=f\big(t,x(t),x(t-1)\big),\\
x(s)=\phi(s)\,, \quad s\in[t_0-1,t_0]
\end{cases}
\quad\text{and}\qquad
\begin{cases}
 y'(t)=f\big(t,y(t),y(t-1)\big),\\
y(s)=\psi(s)\,, \quad s\in[t_0-1,t_0],
\end{cases}
\]
defined on $[t_0,t_0+h]$ for some $h\in[0,1]$.\par
From $\phi(s)\leq \psi(s)$ for each $s\in[t_0-1,t_0]$ and (b) we deduce that $x(t)\leq y(t)$ for $t\in[t_0,t_0+h]$ which together with $x_k(t_0)=a_k=b_k=y_k(t_0)$ yields to
\begin{equation}\label{paraLebesgue}
\frac{1}{h} \int_{t_0}^{t_0+h} f_k(s,x(s),\phi(s-1))\,ds\leq \frac{1}{h}\int_{t_0}^{t_0+h} f_k(s,y(s),\psi(s-1))\,ds\,.
\end{equation}
 Moreover,  denoting by $M(h)=\max_{s\in[t_0,t_0+h]}(\big|x(s)-a|+|\phi(s-1)-c|\big)$ and by $l$ and adequate $l$-bound for $f$, one can easily prove that
\[\left|\frac{1}{h} \int_{t_0}^{t_0+h} \left(f_k(s,x(s),\phi(s-1))-f_k(s,a,c))\right)\,ds\right|\leq M(h) \,\frac{1}{h}\int_{t_0}^{t_0+h}l(s)\,ds\,.
\]
However, thanks to Lebesgue's theorem, for a.e. $t_0\in\R$ the right-hand side of the previous inequality vanishes as $h\to0$ since $\lim_{h\to 0} M(h)=0$. Therefore,
\[\lim_{h\to 0} \frac{1}{h} \int_{t_0}^{t_0+h} f_k(s,x(s),\phi(s-1))\,ds= f_k(t_0,a,c) \qquad\text{for a.e. }t_0\in\R.\]
Analogously, we have that  for a.e. $t_0\in\R$ the integral on the right-hand side of~\eqref{paraLebesgue} converges to $f_k(t_0,b,c)$ as $h\to0$.
Hence, as $h\to 0$, ~\eqref{paraLebesgue} becomes
\[
f_k(t_0,a,c)\le f_k(t_0,b,c)\qquad\text{for a.e. }t_0\in\R,
\]
which concludes the proof of~\ref{Kx}, as claimed.\par
In order to prove~\ref{Ky}, take $f\in E$ and $a,b,c\in\R^N$ such that $b\le c$. Fix $t_0\in\R$ and consider $\phi$, $\psi\in\mathcal{C}([t_0-1,t_0],\R^N)$ defined by
$\phi(s)=(t_0-s)\,b+(s-t_0+1)\,a\,$ and $\psi(s)=(t_0-s)\,c+(s-t_0+1)\,a\,$. Maintaining the above notation for $x(t)$ and $y(t)$ and taking into account that now $x(t_0)=y(t_0)=a$, we deduce~\eqref{paraLebesgue} for all $k\in \{1,\ldots,N\}$, and Lebesgue's theorem provides
\[
f_k(t_0,a,b)\le f_k(t_0,a,c),\qquad \text{for all } k=1,\dots,N,\text{ and a.e. } t_0\in\R\,,
\]
which finishes the proof.\end{proof}
As a consequence of this characterization and of Theorem~\ref{thm:K_y=>cont skew-prod}, one immediately obtains that the monotone skew-product semiflow generated by a suitable set of  Carath\'eodory delay differential equations is continuous.
\begin{cor}
Consider $E\subset\LC(\R^{2N},\R^N)$ invariant with respect to the base flow and with $L^1_{loc}$-equicontinuous $m$-bounds. Let $\overline E$ be its closure in $\LC(\R^{2N},\R^N)$ with respect to $\T\in\{\sigma_{\Theta D}, \T_{\Theta D}\}$, where $\Theta$ is the suitable set of moduli of continuity given by the $m$-bounds in {\rm Remark~\ref{rmk:mod-cont-m-bounds}}, and $D$ a countable dense subset of $\R^N$. If
\begin{equation*}
\begin{split}
\Phi\colon \U\subset \R^+\times\overline E\times\mathcal{C}\to \overline E\times\mathcal{C},\quad (t,f,\phi)\mapsto  \big(f_t, x_t(\cdot,f,\phi)\big),
\end{split}
\end{equation*}
is a monotone skew-product semiflow, then it is also continuous.
\end{cor}
\begin{proof}
The result is a direct consequence of Theorem~\ref{thm:K_y=>cont skew-prod} and Theorem~\ref{thm:monoton<=>Kamke}.
\end{proof}
Next we want to remark that cooperative systems of Carath\'{e}odory differential equations preserve the strict order for each component.
\begin{lem}\label{lema:strict}
 Assume that system~\rm{\eqref{eq:solLCEDDE}} is cooperative, that is, $f$ satisfies  Kamke's conditions~\ref{Kx} and~\ref{Ky}. Consider two initial data $\phi$ and $\psi\in \mathcal{C}$ such that $\phi\leq \psi$ and $\phi_k(0)<\psi_k(0)$ for some $k\in\{1,\ldots,N\}$. Then
\[x_k(t,f,\phi)<x_k(t,f,\psi) \quad \textrm{whenever $t>0$  and  $t\in I_{f,\phi}\cap I_{f,\psi}$}\,.\]
In particular, if $\phi(0)\ll \psi(0)$, then $x(t,f,\phi)\ll x(t,f,\psi)$.
\end{lem}
\begin{proof} We denote $v(t)=x(t,f,\phi)$ and $w(t)=x(t,f,\psi)$ and reason for the first component for simplicity of notation. The proof for the other components is analogous.  We claim that  $w'_1(t)-v'_1(t)\geq -l(t)\,(w_1(t)-v_1(t))$ for an appropriate $l$-bound of $f$. From Theorem~\ref{thm:monoton<=>Kamke} we know that $w(t)-v(t)\geq 0$ for each $t\in I_{f,\phi}\cap I_{f,\psi}$. Then, from~\ref{Kx}   we deduce that
\[f_1(t,w(t),w(t-1))\geq f_1(t,w_1(t),v_2(t),\ldots,v_N(t), v(t-1))\,,
 \]
 which together with the following inequality due to~\ref{L}
\[f_1(t,w_1(t),v_2(t),\ldots,v_N(t), v(t-1))-f_1(t,v(t),v(t-1)\geq -l(t)\,(w(t)-v(t))\]
proves the claim.
Therefore, recalling that $w_1(0)-v_1(0)=\psi_1(0)-\phi_1(0)>0$, and using a comparison result for Carath\'{e}odory differential equations~\cite[Theorem 2]{paper:Walter}, we conclude that
$\;w'_1(t)-v_1'(t)\geq (w_1(0)-v_1(0))\exp\big(-\int_0^t l(s)\,ds\big)>0\;$  for each $t>0$  satisfying  $t\in I_{f,\phi}\cap I_{f,\psi}$\,,
which finishes the proof.
\end{proof}
We conclude this section by considering two more sets of assumptions which still involve Kamke's conditions. Firstly, we will consider~\ref{Kx} and~\ref{Ky} together with a very mild assumption on the Lipschitz coefficients of the components of the vector field, and then a condition of monotonicity which implies both ~\ref{Kx} and ~\ref{Ky}. It is worth noticing that, at least for the strong topologies, we are now able to obtain the same results as before but without assuming that the vector fields have $L^1_{loc}$-equicontinuous $m$-bounds. In fact, we obtain something more. In both cases of strong and weak topologies we are able to prove the continuous dependence of the solutions and the continuity of the associated skew-product semiflow when weaker topologies are employed. To the aim let us recall such two further families of strong and weak topologies on $\SC\big(\R^{2N},\R^N\big)$.
\begin{defn}[Topologies $\T_P$ and $\sigma_P$]\label{def:TP}
Let $P$ be a countable dense subset of $\R^{2N}$. We call $\T_{P}$ (resp.  $\sigma_{P}$) the topology on $\SC\big(\R^{2N},\R^N\big)$ generated by the family of seminorms
\begin{equation*}
p_{I,\, x}(f)=\int_I|f(t,x)|dt \qquad\left(\text{resp. }p_{I,\, x}(f)=\left|\,\int_If(t,x)\, dt\,\right|\right)
\end{equation*}
for $ f\in\SC\big(\R^{2N},\R^N\big)$, $ x\in P,\, I=[q_1,q_2],\,  q_1,q_2\in\Q $.
$\left(\SC\big(\R^{2N},\R^N\big),\T_{P}\right)$ and $\left(\SC\big(\R^{2N},\R^N\big),\sigma_{P}\right)$ are locally convex metric spaces.
\end{defn}
\begin{rmk}\label{rmk:DelayTopologiesChain}
Consider any  two countable dense subsets $D$ and $D_1$ of $\R^N$ and  the countable dense subset $P=D_1\times D$ of $\R^{2N}$. Let  $\Theta$ and $\smash{\widehat\Theta}$ be any pair of suitable sets of moduli of continuity as in Definition~\ref{def:ssmc}, such that for any $I=[q_1,q_2]$, $q_1,q_2\in\Q$ and $j\in\N$ one has
\begin{equation*}
\theta^I_j(t)\le\hat\theta^I_j(t),\qquad\text{for all }t\in[0,\infty).
\end{equation*}
Then, one can draw the following chains of order:
\[
\begin{split}
\sigma_P\le\T_{P}\le\T_{\Theta D}\le\T_{\Theta}\le\T_{\Theta\Theta}\le\T_{\Theta \smash{\widehat\Theta}}\ \text{ and }\ \sigma_P\le\sigma_{\Theta D}\le \sigma_{\Theta}\le\sigma_{\Theta\Theta}\le \sigma_{\Theta \smash{\widehat\Theta}}\le\T_{\Theta \smash{\widehat\Theta}},
\end{split}
\]
where, in particular, the order relations $\, \T_{\Theta}\le\T_{\Theta\Theta}\, $ and $\, \sigma_{\Theta}\le\sigma_{\Theta\Theta}\, $ hold true thanks to Proposition~2.13 in~\cite{paper:LNO3}.  Clearly, one might expand the previous chains of inequalities (or generate new branches) by considering more suitable sets of moduli of continuity (satisfying appropriate relations of partial order) and/or different countable dense subsets of $\R^N$ and $\R^{2N}$, and the corresponding induced topologies.
\end{rmk}
Recall that a  set $S$ of positive functions is \emph{$L^1_{loc}$-bounded} if $\sup_{u\in S}\int_{-r}^r u(t)\,  dt<\infty$ for each $r>0$.\par\smallskip
\begin{defn}\label{def:Lx} A set $E\subset \LC(\R^{2N},\R^N)$ satisfies
\begin{enumerate}[label=\upshape(\textbf{L}$_x$),leftmargin=25pt]
\item\label{Lx} if for each $f\in E$, $j\in\N$ and $k=1,\dots, N$ there is $u^j_{f_k}\in L^1_{loc}$ such that,  for any $(x,y)=(x_1,\dots,x_k,\dots,x_N,y),\;(\overline x^k,y)=(x_1\dots,\overline x_k,\dots,x_N,y)\in B_j\subset\R^{2N}$,
\begin{equation*}
|f_k(t,x,y)-f_k(t,\overline x^k,y)|\le u^j_{f_k}(t)\,|x_k-\overline x_k|, \quad\text{for a.e. }t\in\R\,,
\end{equation*}
 and the set $\big\{u^j_{f_k}\mid f\in E\big\}$ is $L^1_{loc}$-bounded.
\end{enumerate}
\end{defn}
\noindent Notice that this assumption is formulated componentwise and then is weaker that the existence of $L^1_{loc}$-bounded $l$-bounds and $L^1_{loc}$-bounded $l_1$-bounds (see Definitions~\ref{def:LC} and~\ref{def:l1l2bounds}).\par\smallskip
The following result gathers together two auxiliary facts concerning~\ref{Lx}.
\begin{prop}\label{lem:Lx}
Let $P$ be a countable dense subset  of $\R^{2N}$, $E$ be a subset of $\LC(\R^{2N},\R^N)$ satisfying~\ref{Lx}, and $\T\in\{ \T_P,\sigma_P\}$ as in {\rm Definition~\ref{def:TP}}.  Then
\begin{itemize}[leftmargin=19pt,itemsep=2pt]
\item[\rm (i)]  the closure of $E$ with respect to $\T$  satisfies~\ref{Lx};
\item[\rm (ii)] if every function in $E$ satisfies also~\ref{Kx},~\ref{Ky},  and $\Theta, \widehat\Theta$ is any pair of  suitable sets of moduli of continuity as in {\rm Definition~\ref{def:ssmc}}, then
\[
\cls_{(\LC,\sigma_{\Theta\smash{\widehat\Theta}})} (E)=\cls_{(\LC,\sigma_{P})} (E)\quad\text{and}\quad \cls_{(\LC,\T_{\Theta\smash{\widehat\Theta}})} (E)=\cls_{(\LC,\T_{P})} (E).
\]
\end{itemize}
\end{prop}
\begin{proof}
The proof of (i) can be carried out following exactly the same arguments of~\cite[Proposition 2.19]{paper:LNO3}.\par\smallskip
As regards (ii), Proposition~\ref{prop:K_y=>equiv topol} serves as a guide for the proof  and we will only remark the differences between them. Considered $P\subset \R^{2N}$ as in the assumptions, and keeping in mind Remark~\ref{rmk:DelayTopologiesChain}, we shall prove that the convergence in $\sigma_{P}$ implies the one in $\sigma_{\Theta\smash{\widehat\Theta}}$ for any  pair of  suitable sets of moduli of continuity $\Theta, \widehat\Theta$.\par\smallskip
Let $(f_n)_\nin$ be a sequence of functions in $E$ converging to some $f\in\LC(\R^{2N},\R^N)$ with respect to $\sigma_{P}$ and fix any interval $I=[p,q]$, with $p,q\in\Q$ and any $j\in\N$. As in Proposition~\ref{prop:K_y=>equiv topol}, given  $\ep>0$,  $x\in\K^I_j$ and $y\in\widehat\K^I_j$, there is an $m_0\in\N$ for which we construct simple functions $\underline x_P^{m_0}$, $\overline x_P^{m_0}$, $\underline y_{\smash{P}}^{m_0}$, $\overline y_P^{m_0}$ $\colon I \to P$ such that
\[
\underline x_P^{m_0}(t)\le x(t)\le \overline x_P^{m_0}(t)  \quad\text{and}\quad   \underline y_P^{m_0}(t)\le y(t)\le \overline y_P^{m_0}(t) \quad \text{for all }t\in I,
\]
approaching  $x$ and $y$ in $L^\infty(I)$,  uniformly on $\K^I_j$ and  $\widehat\K^I_j$ respectively.
\par\smallskip
Next we can proceed as in the proof of Proposition~\ref{prop:K_y=>equiv topol}.   First we will denote
 \begin{align*}
 \overline x_P^{m_0}(t)&=\big((\overline x_P^{m_0})_1(t),\ldots,(\overline x_P^{m_0})_k(t),\ldots,(x_P^{m_0})_N(t)\big)\,,\\
 (\widehat x_P^{m_0})^k(t)&=\big((\overline x_P^{m_0})_1(t),\ldots,x_k(t),\ldots,(x_P^{m_0})_N(t)\big)\,.
 \end{align*}
From Lemma~\ref{lem:f(t,a,b)->f(t,x(t),y(t))} we deduce that  $(f_n)_k(t,x(t),y(t))\leq (f_n)_k\big(t,(\widehat x_P^{m_0})^k(t), \overline y_P^{m_0}(t)\big)$,
an hence, now instead of~\eqref{eq:16/04-13:27}, we will obtain
\begin{align*}
&\sup_{x\in\K^I_j,\, y\in\widehat \K^I_j}\int _{I_i}\big[\left(f_n\right)_k\!\big(t,x(t), y(t)\big)-f_k\big(t,x(t), y(t)\big)\big]\,dt\\
&\;\;\;  \le \sup_{x\in\K^I_j,\, y\in\widehat \K^I_j}\int _{I_i}\big|\left(f_n\right)_k\!\!\big(t,(\widehat x_P^{m_0})^k(t), \overline y_P^{m_0}(t)\big)-\left(f_n\right)_k\big(t,\overline x_P^{m_0}(t), \overline y_P^{m_0}(t)\big)\!\big|\,dt \\
&\;\;\quad+ \sup_{x\in\K^I_j,\, y\in\widehat \K^I_j}\int _{I_i}\big[\left(f_n\right)_k\big(t,\overline x_P^{m_0}(t), \overline y_P^{m_0}(t)\big)-f_k\big(t,\overline x_P^{m_0}(t), \overline y_P^{m_0}(t)\big)\big]\,dt \\
&\;\;\quad+ \sup_{x\in\K^I_j,\, y\in\widehat \K^I_j}\int _{I_i}\big|f_k\big(t,\overline x_P^{m_0}(t), \overline y_P^{m_0}(t)\big)-f_k\big(t,x(t), y(t)\big)\big|\,dt\\
& \;\;\;\le\sup_{x\in\K^I_j}\int _{I_i}u^{2(j+1)}_{(f_n)_k}(t)\,\|(\overline x_P^{m_0})_k-x_k\|_{L^\infty(I)}\,dt + \frac{3\,\ep}{m_0}\\
& \;\; \quad+\sup_{y\in\widehat \K^I_j}\int _{I_i} \!\! l^{2(j+1)}_f(t)\,\|\overline y_P^{m_0}-y\|_{L^\infty(I)}dt \le\frac{3\,\ep}{m_0}+3\,\ep\int _{I_i} \!\! \left[u^{2(j+1)}_{(f_n)_k}(t)+l^{2(j+1)}_f(t)\right]dt.
\end{align*}
The rest of the proof is exactly as in Proposition~\ref{prop:K_y=>equiv topol}, once one realizes that, due to the assumption of  $L^1_{loc}$-boundedness for $\big\{u^{2(j+1)}_{g}\mid g\in E\big\}$, the integral on the right-hand side of the previous inequality is bounded uniformly on $\nin$. The proof for $\T_{P}$ and $\T_{\Theta\smash{\widehat\Theta}}$ is similar and again on the line of the suggestion given in  Proposition~\ref{prop:K_y=>equiv topol}.
\end{proof}
\par
We are now ready to give the first result of continuous dependence of the solutions employing the topologies  $\T_P$ and $\sigma_P$, and of continuity of the induced skew-product semiflows.
\begin{thm}\label{thm:K_xK_y+lipschitz=>cont skew-prod}
Let $E$ be a subset of $\LC(\R^{2N},\R^N)$ satisfying~\ref{Lx}, and such that each element of $ E$ satisfies~\ref{Kx} and~\ref{Ky}.
Let $P$ be a countable dense subset of $\R^{2N}$ and consider the topologies $\sigma_{P}, \T_{ P}$ as constructed in {\rm Definition~\ref{def:TP}}.
\begin{itemize}[leftmargin=20pt,itemsep=2pt]
\item[\rm (i)]   If $\big(\phi_{n}\big)_\nin$ is a sequence in $\mathcal{C}$ converging uniformly to $\phi\in \mathcal{C}$, and $(f_n)_\nin$ is a sequence in $E$ converging to $f$ in $(\LC(\R^{2N},\R^N),\T_{P})$, then, with the notation of {\rm Theorem~\ref{eq:solLCEDDE}}, one has that
\begin{equation*}
 x(\cdot,f_n,\phi_{n}) \xrightarrow{\nti}  x(\cdot,f,\phi)
\end{equation*}
uniformly in any $[-1,T]\subset I_{f,\phi}$.
\item[\rm (ii)] if $E$ is invariant with respect to the base flow, and $\overline E$ denotes the closure of $E$ in $(\LC(\R^{2N},\R^N),\T_{P})$ then  the map
\begin{equation*}
\begin{split}
\quad\Phi\colon \U\subset \R^+\times\overline E\times\mathcal{C}\to \overline E\times\mathcal{C},\quad (t,f,\phi)\mapsto  \big(f_t, x_t(\cdot,f,\phi)\big),
\end{split}
\end{equation*}
is a local monotone continuous skew-product semiflow.
\item[\rm (iii)] If additionally $E$ has $L^1_{loc}$-equicontinuous $m$-bounds, then {\rm (i)} and  {\rm (ii)} hold true also when the topology $\sigma_{P}$ is employed in place of $\T_{P}$.
\end{itemize}
\end{thm}
\begin{proof}
In order to prove (i), first consider $t\in[0,1]$ and  the Carath\'eodory ordinary differential equations Cauchy problems
\begin{equation}\label{eq:16/04-17:57}
\begin{cases}
 x'(t)= f_n\big(t,x(t),\phi_n(t-1)\big)\\
x(0)=\phi_n(0).
\end{cases}\quad\text{and}\qquad
\begin{cases}
 x'(t)= f\big(t,x(t),\phi(t-1)\big)\\
x(0)=\phi(0).
\end{cases}
\end{equation}
We shall prove the uniform convergence of $\big(x(\cdot,f_n,\phi_n)\big)_\nin$ to  $x(\cdot,f,\phi)$ in $[0,T]$ for any $0< T<\min\{1,b_{f,\phi}\}$.
Denote
\begin{equation}
0<\rho =1+ \max\big\{ (\|\phi_n\|_{\mathcal{C}})_\nin,\  \|x(\cdot,f,\phi)\|_{\mathcal{C}([-1,T])}\big\}\, ,
\label{eq:DefRHO}
\end{equation}
and for all $\nin$ define
\begin{equation*}
z_n(t)=
\begin{cases}
x(t,f_n,\phi_n), & \text{if $0\le t< T_n$,} \\
x(T_n,f_n,\phi_n), & \text{if $T_n\le t\le T$.}
\end{cases}
\end{equation*}
where $T_n=\sup\{t\in[0,T]\mid|x(s,f_n,\phi_n)|\le \rho,\, \forall\, s\in[0,t]\}$. From~\eqref{eq:DefRHO} and the continuity of $\big(x(\cdot,f_n,\phi_n)\big)_\nin$, we deduce that $T_n>0$ for any $n\in\N$. In particular, $(z_n)_{\nin}$ is uniformly bounded. Consider $j\in\N$ so that $2\rho<j$.  Now we want to control $|z_n(t_1)- z_n(t_2)|$ for $t_1$, $t_2\in[0,T_n)$ and $t_1<t_2$.  For simplicity of notation, we will just study the first of the components:
\[|(z_n)_1(t_1)-(z_n)_1(t_2)|\le \int_{t_1}^{t_2}\left|(f_n)_1\big(s,z_n(s),\phi_n(s-1)\big)\right|\, ds\,.\]
It is easy to check that there is a $k_j>j$ and $a$, $b\in P \cap B_{k_j}$ such that  $a\leq (x,y)\leq b$  for each $(x,y)\in B_j$.
From~\ref{Kx} and~\ref{Ky} we obtain
\begin{multline*}
 (f_n)_1(s,(z_n)_1(s),a_2,\ldots,a_{2N})\leq (f_n)_1(s,z_n(s),\phi_n(s-1))\\ \leq (f_n)_1(s,(z_n)_1(s),b_2,\ldots,b_{2N})\,.
\end{multline*}
Thus, denoting by $g,\;g_n,\;h$ and $h_n\colon \R\to\R$ the functions of  $\LC(\R,\R)$ given by
\begin{align*}
g(t,v)=(f)_1(t, v,a_2\dots, a_{2N}) \quad  \text{and}& \quad g_n(t,v)=(f_n)_1(t, v,a_2\dots, a_{2N})\,\\
h(t,v)=(f)_1(t, v,b_2\dots, b_{2N})  \quad \text{and}& \quad h_n(t,v)=(f_n)_1(t, v,b_2\dots, b_{2N})\,,
\end{align*}
we deduce that $\left|(f_n)_1(s,z_n(s),\phi_n(s-1))\right|\leq \left| g_n\big(s,(z_n)_1(s)\big)\right|+\left|h_n\big(s, (z_n)_1(s)\big)\right|$ and
\begin{align}\nonumber
|(z_n)_1(t_1)-(z_n)_1(t_2)|&\le\int_{t_1}^{t_2}\left|  g_n\big(s, (z_n)_1(s)\big)- g\big(s, (z_n)_1(s)\big)\right|ds\\\label{desi:gnhn}
&\qquad+ \int_{t_1}^{t_2}\!\left| h_n\big(s, (z_n)_1(s)\big)- h\big(s, (z_n)_1(s)\big)\right|ds\\ \nonumber
&\qquad+2\int_{t_1}^{t_2}m^{k_j}_f(s)\,ds= I_1+I_2+ I_3
\end{align}
where $m^{k_j}_f$ is the optimal $m$-bound for $f$ on $B_{k_j}\subset\R^{2N}$. \par
Let $D_1$ and $D_2$ be countable dense subsets of $\R$ and $\R^{2N-1}$ such that $a_1$, $b_1\in D_1$ and $(a_2,\ldots,a_{2N})$, $(b_2,\ldots,b_{2N})\in D_2$. Since $(f_n)_\nin$ converges to $f$ with respect to $\T_P$ and the assumptions of Proposition~\ref{lem:Lx} are satisfied, then  $(f_n)_\nin$ converges to $f$  also with respect to $\T_{D_1\times D_2}$. From this fact one easily obtains that $(g_n)_\nin$ converges to $g$, and $h_n$ converges to $h$ with respect to $\T_{D_1}$.\par\smallskip
Moreover, from~\ref{Lx} the sequence of functions $\{g_n\mid \nin\}$  and $\{h_n\mid \nin\}$ have $L^1_{loc}$-bounded $l$-bounds, and consequently,
thanks to Theorem 4.12 in~\cite{paper:LNO1}  they also converge to $g$ and $h$ respectively with respect to the topology $\T_B$. In particular, this fact implies that the integrals $I_1$ and $I_2$ of~\eqref{desi:gnhn} tend to $0$ as $n$ goes to $\infty$ (see~\cite{paper:LNO1} for details about the convergence in $\T_{B}$).  On the other hand, thanks to  the absolute continuity of the Lebesgue integral, given $\ep>0$ there is a $\delta>0$ such that $I_3\leq \ep$ provided that $t_2-t_1<\delta$.\par\smallskip
Summing up, and taking into account that in $[T_n,T]$ the left-hand side of~\eqref{desi:gnhn} vanishes, given $\ep>0$ there is a $\delta>0$ and an $n_1\in\N$ such that
\[ |(z_n)_1(t_1)-(z_n)_1(t_2)|\leq 3\,\ep \;\; \text{whenever } n\geq n_1,\; 0\le t_1 \le t_2\le T \; \text{and } t_2-t_1<\delta\,,\]
from which the equicontinuity of the sequence $(z_n)_{\nin}$ is obtained.
Therefore, Ascoli-Arzel\'a's theorem implies that $(z_n)_{\nin}$ converges uniformly, up to a subsequence, to some continuous function $z\colon[0,T]\to \R^N$. \par\smallskip
 Next, we prove that $z(t)\equiv x(t,f,\phi)$  for each  $t\in[0,T]$. Define
\begin{equation}
T_0=\sup \left\{t\in[0,T]\mid |z(s)|<\rho-1/2\quad \text{for all } s\in[0,t]\right\}\, ,
\label{eq:defT0}
\end{equation}
and notice that $T_0>0$ because  $(\phi_{n})_\nin$ converges uniformly to $\phi$ in $[-1,0]$ and $z$ is continuous.  Since $z_n$ converges uniformly to $z$ in $[0,T]$, then there exists an $n_0\in\N$ such that if $n>n_0$, then $|z_n(t)|<\rho-1/4$ for all $ t\in[0,T_0]$.

Therefore, for any $t\in[0,T_0]$  and for any $n>n_0$ one has  $z_n(t)=x(t,f_n,\phi_n)$ and
\begin{equation}\label{eq:14.06-19:35}
z_n(t)=\phi_n(0)+\int_0^tf_n\big(s,z_n(s),\phi_n(s-1)\big)\, ds\, ,\qquad t\in[0,T_0]\, ,\ n>n_0\, .
\end{equation}
Now consider the compact set $\mathcal{K}=\{z_n\mid\nin\}\cup\{z\}\subset \mathcal{C}\big([0,T],\R^N\big)$ and let $\theta\colon\R^+\to\R^+$ be its modulus of continuity. Moreover, since $\big(\phi_{n}\big)_\nin$ converges  uniformly to $\phi\in \mathcal{C}$, the set $\{\phi_n\mid\nin\}\cup\{\phi\}$ also has a common modulus of continuity, say $\hat\theta\colon\R^+\to\R^+$. The sets $\Theta=\{\theta\}$ and $\widehat\Theta=\{\hat\theta\}$ are suitable sets of moduli of continuity. Then again thanks to  Proposition~\ref{lem:Lx}, since $(f_n)_\nin$ converges to $f$ in $(\LC(\R^{2N},\R^N),\T_{P})$, then  $(f_n)_\nin$ converges to $f$  also with respect to $\T_{\Theta\smash{\widehat\Theta}}$, and passing to the limit  as $\nti$ in \eqref{eq:14.06-19:35}, we deduce that
\begin{equation*}
z(t)=x_0+\int_0^tf\big(s,z(s),z(s-1)\big)\, ds\, ,\qquad t\in[0,T_0]\, .
\end{equation*}
In other words, in $[0,T_0]$, $z(t)$ is the solution of the problem on the right-hand side of \eqref{eq:16/04-17:57}. Finally, notice that it must be $T_0=T$. Otherwise, by \eqref{eq:defT0} and by the continuity of $z$, one would have $|z(T_0)|=|x(T_0,f,x_0)|=\rho-1/2$, which contradicts \eqref{eq:DefRHO}. Hence, for any  $t\in [0,T]$ we have that $x(t,f,x_0)= z(t)$ and $x(t,f_n,x_{0,n})= z_n(t)$ for any $\nin$. This reasoning can be iterated to obtain the uniform convergence on any compact interval of the maximal interval of definition of $x(\cdot,f,x_0)$, which concludes the proof of  (i). \par\smallskip
We obtain (ii) as a consequence of (i), of Propositions~\ref{prop:Kx-Ky propagate to closure} and~\ref{lem:Lx},  and of the continuity of the function $(t,f)\mapsto f_t$ proved in~\cite[Theorem~3.8]{paper:LNO3}.\par\smallskip
As regards (iii), one can proceed exactly as for (i) except that now the equicontinuity of the sequence $(z_n)_\nin$ follows from the $L^1_{loc}$-equicontinuity of the $m$-bounds. Notice that in this case, the limit as $\nti$ in \eqref{eq:14.06-19:35} is initially taken only for $t\in\Q$ but  then one gets the result for any $t\in[0,T_0]$ thanks to the continuity of $z$ and of the integral operator. The result for the skew-product semiflow is obtained once again combining the continuous dependence of the solutions with the continuity of the function $(t,f)\mapsto f_t$ proved in~\cite[Theorem~3.8]{paper:LNO3}.
\end{proof}
We introduce a stronger monotonicity condition implying both~\ref{Kx} and~\ref{Ky}.
\begin{defn} A function $ f\in\LC(\R^{2N},\R^N)$ satisfies
\begin{enumerate}[label=\upshape(\textbf{K}$_{xy}$),leftmargin=31pt]
\item\label{Kxy} if for any  $a,b,c,d\in\R^N$ with $a\le b$ and $c\le d$, then
\[
f_k(t,a,c)\le f_k(t,b,d),\qquad \text{for all } k=1,\dots,N,\text{ and a.e. } t\in\R.\\[1ex]
\]
\end{enumerate}
\end{defn}
The following result gathers together some of the analogous properties proved for~\ref{Ky} before.
\begin{prop}\label{lem:Kxy}
Let $P$ be a countable dense set  of $\,\R^{2N}$ and $\T\in\{ \T_P,\sigma_P\}$.
\begin{itemize}[leftmargin=19pt,itemsep=2pt]
\item[\rm (i)] If  $f\in\LC(\R^{2N},\R^N)$ satisfies ~\ref{Kxy}, then  for any interval $I\subset\R$ and any $x,y,u, v\in \mathcal{C}(I,\R^N)$, with $x(t)\le u(t)$ and $y(t)\le v(t)$ for all $t\in I$,  one has that
\begin{equation*}
f\big(t,x(t),y(t)\big)\le f\big(t,u(t),v(t)\big),\qquad \text{for a.e. } t\in I.
\end{equation*}
\item[\rm (ii)] If every function in $E\subset\LC(\R^{2N},\R^N)$ satisfies~\ref{Kxy}, then also any function in the closure of $E$ with respect to $\T$ satisfies~\ref{Kxy}.
\item[\rm (iii)] If every function in $E\subset\LC(\R^{2N},\R^N)$ satisfies~\ref{Kxy},  and $\Theta, \widehat\Theta$ is any pair of  suitable sets of moduli of continuity as in {\rm Definition~\ref{def:ssmc}}, then
\[
\cls_{(\LC,\sigma_{\Theta\smash{\widehat\Theta}})} (E)=\cls_{(\LC,\sigma_{P})} (E)\quad\text{and}\quad \cls_{(\LC,\T_{\Theta\smash{\widehat\Theta}})} (E)=\cls_{(\LC,\T_{P})} (E).
\]
\end{itemize}
\end{prop}
\begin{proof}
The statements can be proved using arguments similar to the ones employed in the proofs of Lemma~\ref{lem:f(t,a,b)->f(t,x(t),y(t))}, Proposition~\ref{prop:Kx-Ky propagate to closure} and Proposition~\ref{prop:K_y=>equiv topol}, respectively.
\end{proof}
Then the following additional result of continuity of the skew-product semiflows generated by a set of Carath\'eodory delay differential equations are obtained.
\begin{thm}\label{thm:K_xy=>cont skew-prod}
Consider $E\subset\LC(\R^{2N},\R^N)$ such that every element of $ E$ satisfies~\ref{Kxy}. Let $P$ be any countable dense subset of $\R^{2N}$  and $\sigma_{P},\; \T_{ P}$  the topologies introduced in {\rm Definition~\ref{def:TP}}.
\begin{itemize}[leftmargin=20pt,itemsep=2pt]
\item[\rm (i)]   If $\big(\phi_{n}\big)_\nin$ is a sequence converging uniformly to $\phi$ in $\mathcal{C}$, and $(f_n)_\nin$ is a sequence in $E$ converging to $f$ in $(\LC(\R^{2N},\R^N),\T_{P})$, then, with the notation of {\rm Theorem~\ref{eq:solLCEDDE}}, one has that
\begin{equation*}
 x(\cdot,f_n,\phi_{n}) \xrightarrow{\nti}  x(\cdot,f,\phi)
\end{equation*}
uniformly in any $[-1,T]\subset I_{f,\phi}$.
\item[\rm (ii)]  if $E$ is invariant with respect to the base flow, and $\overline E$ denotes the closure of $E$ in $(\LC(\R^{2N},\R^N),\T_{P})$ then  the map
\begin{equation*}
\qquad\Phi\colon \U\subset \R^+\times\overline E\times\mathcal{C}\to \overline E\times\mathcal{C},\quad (t,f,\phi)\mapsto  \big(f_t, x_t(\cdot,f,\phi)\big),
\end{equation*}
is a local monotone continuous skew-product semiflow
\item[\rm (iii)] If additionally $E$ has $L^1_{loc}$-equicontinuous $m$-bounds, then  {\rm (i)} and  {\rm (ii)} hold true also when the topology $\sigma_{P}$ is employed in place of $\T_{P}$.
\end{itemize}
\end{thm}
\begin{proof}
The proof can be carried out reasoning  as in Theorem~\ref{thm:K_xK_y+lipschitz=>cont skew-prod}. Maintaining the notation of its proof, from~\ref{Kxy} we deduce the inequality
\begin{align}\nonumber
|(z_n)_k(t_1)-(z_n)_k(t_2)|& \le \int_{t_1}^{t_2}\big|(f_n)_k\big(s,z_n(s), \phi_n(s-1)\big)\big|\, ds\\  \label{eq:desi3.16}
& \le \int_{t_1}^{t_2}\big|(f_n)_k\big(s,a\big)-(f)_k\big(s,a\big)  \big|\, ds\\ \nonumber
&\qquad+\int_{t_1}^{t_2}\big|(f_n)_k\big(s,b\big)-(f)_k\big(s,b\big)  \big|\, ds+2\int_{t_1}^{t_2}m^{j+1}_f(s)\,ds\,,
\end{align}
for $k=1,\ldots,N$, which, from  the convergence of  $(f_n)_\nin$ to $f$ with respect to $\T_P$ and the absolute continuity of the Lebesgue integral, leads to the equicontinuity of $(z_n)_{\nin}$. The rest of the proof follows exactly the same arguments with the help, in this case, of Proposition~\ref{lem:Kxy}.
\end{proof}
The reasoning behind the proofs of Theorems~\ref{thm:K_xy=>cont skew-prod} and~\ref{thm:K_xK_y+lipschitz=>cont skew-prod} immediately suggests that if one deals with a set of ordinary differential equation of Carath\'eodory type, two new results of continuity for the associated skew-product flow can be derived.
Given $f\in \LC(\R^{N},\R^N)$ and $x_0\in\R^N$, we will denote by $x(\cdot,f,x_0): I_{f,x_0}\to\R^N$ the unique maximal solution of the Cauchy problem
\[  x'(t)=f\big(t,x(t)\big),\; x(0)=x_0\,.\]
The monotoniticy conditions, as well as the corresponding condition~\ref{Lx} for a set of functions in this case, are stated in the next definitions.
\begin{defn}
We say that $ f\in\LC(\R^{N},\R^N)$ satisfies
\begin{enumerate}[label=\upshape(\textbf{K$_\arabic*$}),leftmargin=27pt]
\item\label{K1} if for any $a,b\in\R^N$ with $a\le b$ and $a_k=b_k$ for some $k\in\{1,\dots,N\}$, then
\[
f_k(t,a)\le f_k(t,b),\qquad \text{for a.e. } t\in\R;\\[1ex]
\]
\item\label{K2} if for any $a,b\in\R^N$ with $a\le b$, then
\[
f_k(t,a)\le f_k(t,b),\qquad \text{for all } k=1,\dots,N,\text{ and a.e. } t\in\R.\\[1ex]
\]
\end{enumerate}
\end{defn}
\begin{defn} A set $E\subset \LC(\R^{N},\R^N)$ satisfies
\begin{enumerate}[label=\upshape(\textbf{L}$_1$),leftmargin=27pt]
\item\label{LxODE} if for each $f\in E$, $j\in\N$ and $k=1,\dots, N$ there is $u^j_{f_k}\in L^1_{loc}$ such that,  for any $x=(x_1,\dots,x_k,\dots,x_N)$, $\overline x^k=(x_1\dots,\overline x_k,\dots,x_N)\in B_j\subset\R^{N}$,
\begin{equation*}
|f_k(t,x)-f_k(t,\overline x^k)|\le u^j_{f_k}(t)\,|x_k-\overline x_k|, \quad\text{for a.e. }t\in\R\,,
\end{equation*}
 and the set $\{u^j_{f_k}\mid f\in E\}$ is $L^1_{loc}$-bounded.
\end{enumerate}
\end{defn}
The proofs of the following continuity results are omitted as they can be easily worked out from the ones of Theorems~\ref{thm:K_xK_y+lipschitz=>cont skew-prod}
and~\ref{thm:K_xy=>cont skew-prod}.
\begin{prop}\label{thm:Kx+LipschitzODEs-cont skew-prod}
Consider $E\subset\LC(\R^{N},\R^N)$ satisfying~\ref{LxODE} and such that every element of $ E$ satisfies~\ref{K1}. Let $D$ be a countable dense subset of $\R^{N}$.
\begin{itemize}[leftmargin=20pt,itemsep=2pt]
\item[\rm (i)]   Let $(f_n)_\nin$ be a sequence in $E$ converging to $f$ in $(\LC(\R^{N},\R^N),\T_{D})$ and $(x_{0,n})_\nin$ is a sequence converging to $x_0\in\R^N$, then
\[ \qquad x(\cdot,f_n,x_{0,n}) \xrightarrow{\nti}  x(\cdot,f,x_0)\]
     uniformly in any $[T_1,T_2]\subset I_{f,x_0}$.
\item[\rm (ii)]   If $E$ is invariant with respect to the base flow, and $\overline E$ denotes the closure of $E$ in $(\LC(\R^{N},\R^N),\T_D)$, then
\[
\Phi\colon \U\subset \R\times\overline E\times \R^N\to \overline E\times\R^N,\quad (t,f,x_0)\mapsto  \big(f_t, x(t,f,x_0)\big),
\]
is a local monotone continuous skew-product flow.
\item[\rm (iii)] If additionally $E$ has $L^1_{loc}$-equicontinuous $m$-bounds, then {\rm (i)} and  {\rm (ii)} hold true also when the topology $\sigma_{D}$ is employed in place of $\T_{D}$.
\end{itemize}
\end{prop}
\begin{prop}\label{thm:KxyODEs-cont skew-prod}
 Consider $E\subset\LC(\R^{N},\R^N)$ such that every element of $ E$ satisfies~\ref{K2} and let $D$ be a countable dense subset of $\R^{N}$.
\begin{itemize}[leftmargin=20pt,itemsep=2pt]
\item[\rm (i)]   Let $(f_n)_\nin$ be a sequence  in $E$ converging to $f$ in $(\LC(\R^{N},\R^N),\T_{D})$ and $(x_{0,n})_\nin$ is a sequence converging to $x_0\in\R^N$, then
\[ \qquad x(\cdot,f_n,x_{0,n}) \xrightarrow{\nti}  x(\cdot,f,x_0)\]
     uniformly in any $[T_1,T_2]\subset I_{f,x_0}$.
\item[\rm (ii)] If $E$ is invariant with respect to the base flow, and $\overline E$ denotes the closure of $E$ in $(\LC(\R^{N},\R^N),\T_D)$, then
\[
\Phi\colon \U\subset \R\times\overline E\times \R^N\to \overline E\times\R^N,\quad (t,f,x_0)\mapsto  \big(f_t, x(t,f,x_0)\big),
\]
is a local monotone continuous skew-product flow.
\item[\rm (iii)] If additionally $E$ has $L^1_{loc}$-equicontinuous $m$-bounds, then {\rm (i)} and  {\rm (ii)} hold true also when the topology $\sigma_{D}$ is employed in place of $\T_{D}$.
\end{itemize}
\end{prop}
\section{Monotone sublinear Carath\'{e}odory skew-product semiflows}\label{sec:sublinear}
Let $E$ be a subset of $\LC(\R^{2N},\R^N)$ invariant with respect to the base flow  and closed with respect to a topology $\T$.  In the previous section, assuming appropriate conditions on $E$  and $\T$, we obtained several results asserting that the map
\begin{equation}\label{monosemitau}
\Phi \colon \R^+\times E\times\mathcal{C}\to E\times\mathcal{C},\quad (t,f,\phi)\mapsto  \big(f_t, x_t(\cdot,f,\phi)\big)
\end{equation}
is a monotone continuous skew-product semiflow. This will be the situation considered throughout the current section. Recall that $\mathcal{C}$  is endowed with a strong partial ordering~\eqref{orderingC} with positive cone $\mathcal{C}^+=\{\phi\in \mathcal{C}\mid \phi\geq 0\}$. \par
The following definitions, previously introduced both for deterministic and random monotone semiflows (see~\cite{paper:nno2005} and~\cite{book:chue}), are of great importance for the long-term behavior of the trajectories. Note that these concepts are associated to the semiflow $\Phi$,
although this dependence does not appear explicitly in the
definition.
\begin{defn}[semi-equibrium]\label{defn:semi-equilibria}
A map $a\colon E\to \mathcal{C}$ such that $x(t,f,a(f))$ is defined for any
$t\geq 0$~is
\begin{itemize}[leftmargin=25pt,itemsep=3pt]
\item[\rm (i)] a \emph{sub-equilibrium}  if $a(f_t)\le x_t(\cdot,f,a(f))$  for any $f\in E$ and $t\ge0$,
\item[\rm (ii)] a \emph{super-equilibrium}  if $x_t(\cdot,f,a(f))\leq a(f_t)$  for any $f\in E$ and $t\ge0$,
\item[\rm (iii)] an \emph{equilibrium} if $a(f_t)=x_t(\cdot,f,a(f))$ for any $f\in E$ and $t\ge0$.
\end{itemize}
 We will use the term {\em semi-equilibrium} to
refer either to a super or a
sub-equilibrium.
\end{defn}
\begin{defn}\label{def:semicontinuity}
A sub-equilibrium (resp. super-equilibrium) $a\colon E\to\mathcal{C}$ is  \emph{lower-semicontinuous} (resp. \emph{upper-semicontinuous)} if for any sequence $(f_n)_\nin$ converging to some $f$ in $E$ with respect to $\T$:
\begin{itemize}[leftmargin=25pt,itemsep=3pt]
\item[\rm (i)] $\{a(f_n)\mid\nin\}$ is a relatively compact subset of $\mathcal{C}$, and
\item[\rm (ii)]$a(f)\le \phi$ (resp. $\phi\le a(f)$) whenever  there is a subsequence $(f_{n_k})_{k\in\N}$ satisfying $\lim_{k\to\infty}a(f_{n_k})=\phi$.
\end{itemize}
These concepts apply to the case of an equilibrium because it is both a sub- and a super-equilibrium.
\end{defn}
The notion of semicontinuous semiequilibrium introduced above will be enough for the purposes of this work. Anyhow, it is easy to check that if $a\colon E\to\mathcal{C}$ satisfies Definition~\ref{def:semicontinuity}, then it is possible to define an associated multievaluated function $\mathfrak{a}\colon E\to\mathcal{C}$ which is upper-semicontinuous in the usual sense (see Aubin and Frankowska \cite{book:AF}.
\begin{rmk}
Notice that if $a$ is a lower-semicontinuous sub-equilibrium, then the set $\{(f,\phi)\in E\times \mathcal{C} \mid  a(f) \le \phi\}$ is a positively invariant closed set. Indeed, the positively invariance follows from Definition~\ref{defn:semi-equilibria}(i). As regards the closeness, consider a sequence $(f_n,\phi_n)_\nin$ converging to some $(f,\phi)\in E\times\mathcal{C}$ and such that for each $\nin$, $a(f_n)\leq \phi_n$. By the lower-semicontinuity there is a subsequence $(f_{n_k})_{k\in\N}$ and $\psi\in\mathcal{C}$ such that $\lim_{k\to\infty}a(f_{n_k})=\psi$ with $a(f)\leq \psi$, and  hence, $a(f)\leq \psi\le \phi$, as claimed. Similarly, $\{(f,\phi)\in E\times\mathcal{C}\mid b(f)\ge \phi\}$  is a positively invariant closed set, provided that $b$ is an upper-semicontinuous super-equilibrium.
\end{rmk}
\begin{thm}\label{thm:equilibria} Consider $\Phi$  one of the monotone continuous skew-product semiflows given in~\eqref{monosemitau}.
  Let $a\colon E\to\mathcal{C}$ be a lower-semicontinuous sub-equilibrium and let $b\colon E\to\mathcal{C}$  be an upper-semicontinuous super-equilibrium such that $a(f)\le b(f)$ for each $f\in E$, then
\begin{itemize}[leftmargin=25pt,itemsep=2pt]
\item[\rm (i)]  for any $\tau\ge0$ the function
\begin{equation*}
a_\tau\colon E\to \mathcal{C} ,\quad
f\mapsto a_\tau(f)=x_\tau(\cdot,f_{-\tau},a(f_{-\tau}))
\end{equation*}
is a lower-semicontinuous sub-equilibrium  satisfying
\begin{equation}\label{eq:inefora}
a(f)\le a_{\tau_1}(f)\le a_{\tau_2}(f)\le b(f) \quad \text{whenever}\;\; 0\le \tau_1\le \tau_2 \text{ and } f\in E;
\end{equation}
\item[\rm (ii)] the limit $u(f)=\lim_{\tau\to\infty}a_{\tau}(f)$
exists for any $f\in E$, and $u\colon E\to\mathcal{C}$ is a lower-semicontinuous equilibrium;
\item[\rm (iii)] for any $\tau>0$ the function
\[
b_\tau\colon E\to \mathcal{C} ,\quad
f\mapsto b_\tau(f)=x_\tau(\cdot,f_{-\tau},b(f_\tau)),
\]
is an upper-semicontinuous super-equilibrium  satisfying
\begin{equation}\label{eq:ineforb}
a(f)\le b_{\tau_2}(f)\le b_{\tau_1}(f)\le b(f) \quad \text{whenever}\;\; 0\le \tau_1\le \tau_2 \text{ and } f\in E;
\end{equation}
\item[\rm (iv)] the limit $ v(f)=\lim_{\tau\to\infty}b_{\tau}(f) $
exists  for any $f\in E$, and $v\colon E\to\mathcal{C}$ is an  upper-semicontinuous equilibrium;
\item[\rm (v)] for any $f\in E$, $a(f)\le u(f)\le v(f)\le b(f)$.
\end{itemize}
\end{thm}
\begin{proof}
(i) See~\cite[Proposition 3.4.1]{book:chue} to prove that $a_\tau$ is a sub-equilibrium satisfying~\eqref{eq:inefora} for each fixed $\tau> 0$. Next we show that it is lower-semicontinuous. First we check that given a sequence $(f_n)_\nin$ converging to some $f\in E$ with respect to $\T$, the set
$\{a_\tau(f_n)\mid\nin\}$ is relatively compact in $\mathcal{C}$, i.e. there is a convergent subsequence. From  $\lim_{n\to\infty}(f_n)_{-\tau}= f_{-\tau}$ and the lower semicontinuity  of $a$  we deduce that there is a subsequence $(f_{n_k})_{k\in\N}$  and $\phi^\tau\in \mathcal{C}$ such that
\[\lim_{k\to\infty} a((f_{n_k})_{-\tau})=\phi^\tau \quad  \text{with }  a(f_{-\tau})\leq \phi^\tau\,.\]
Then from the continuity of the skew-product semiflow~\eqref{monosemitau} we conclude that
\[ a_\tau(f_{n_k})=x_\tau(\cdot, (f_{n_k})_{-\tau}, a((f_{n_k})_{-\tau}))\xrightarrow{k\to\infty} x_\tau(\cdot,f_{-\tau}, \phi^\tau) \quad \text{in } \mathcal{C}\,,\]
as claimed. Finally, $a_\tau(f)=x_\tau(\cdot, f_{-\tau}, a(f_{-\tau}))\leq x_\tau(\cdot, f_{-\tau},\phi^\tau)$ follows from the monotone character of the skew-product semiflow~\eqref{monosemitau}, and hence, the lower-semicontinuity of $a_\tau$ is obtained.\par\smallskip
(ii)  First we claim that, for each $f\in E$,  the set  $\{ a_\tau(f)\mid \tau \geq 1\}\subset \mathcal{C}$ is relatively compact. The boundedness follows from~\eqref{eq:inefora}, and the equicontinuity from the cocycle property
\[a_\tau(f) = x_\tau(\cdot, f_{-\tau}, a(f_{-\tau}))\!=\! x_1(\cdot,f_{-1}, x_{\tau-1}(\cdot,f_{-\tau}, a(f_{-\tau}))) = x_1(\cdot,f_{-1},a_{\tau-1}(f_{-1}))
\]
and the $m$-bound for $f_{-1}$ because $\{a_{\tau-1}(f_{-1})\mid \tau\geq 1\}$ is also bounded.
As a consequence, together  with the monotonicity property~\eqref{eq:inefora}, we deduce  that there exist a unique $u(f)=\lim_{\tau\to\infty}a_{\tau}(f)=\sup_{\tau\geq 0} a_\tau(f)$ in $\mathcal{C}$, as stated. In order to check that $u$ is an equilibrium, notice that
\begin{align*}
x_t(\cdot,f,u(f))=&\lim_{\tau\to\infty} x_t(\cdot,f, a_\tau(f))=\lim_{\tau\to\infty} x_t(\cdot,f, x_\tau(\cdot, f_{-\tau}, a(f_{-\tau})))\\
= & \lim_{\tau\to\infty} x_{t+\tau}(\cdot, f_{-\tau}, a(f_{-\tau}))=\lim_{\tau\to\infty} a_{t+\tau}(f_t)=u(f_t)
\end{align*}
whenever $t\ge 0$ and $f\in E$.
Finally, we prove that it is lower-semicontinuous. Let $(f_n)_\nin$ be a sequence converging to some $f\in E$ with respect to $\T$. We check that $\{u(f_n)\mid \nin\}$ is relatively compact.  The boundedness follows from~\eqref{eq:inefora}, the lower-semicontinuity of $a$ and the upper-semicontinuity of $b$. Next notice that, since $u$ is an equilibrium, we have
\[u(f_n)=x_1(\cdot, (f_n)_{-1}, u((f_n)_{-1}))=x(1+\cdot, (f_n)_{-1}, u((f_n)_{-1}))\quad \text{for each } \nin, \]
 and then,  if we are in one of the cases for which $E$ has $L^1_{loc}$-equicontinuous $m$-bounds, the equicontinuity follows from the boundedness of the set of initial data $\{u((f_n)_{-1}) \mid\nin\}$. When the $L^1_{loc}$-equicontinuity of the $m$-bounds is not verified, that is, Theorem~\ref{thm:K_xK_y+lipschitz=>cont skew-prod}[(i)-(ii)]  and Theorem~\ref{thm:K_xy=>cont skew-prod}[(i)-(ii)], the equicontinuity follows from similar arguments to the ones applied to check the equicontinuity of the sequence $(z_n)_{\nin}$ in these theorems, see inequalities~\eqref{desi:gnhn} and~\eqref{eq:desi3.16}, respectively. Therefore,  $\lim_{k\to\infty} u(f_{n_k})=\phi$ for some $\phi$, and from the lower-semicontinuity of $a_\tau$, up to a subsequence, there is  a $\psi^\tau\in\mathcal{C}$  with
 \[a_\tau(f)\leq \psi^\tau=\lim_{k\to\infty}a_\tau(f_{n_k})\quad \text{for each } \tau\geq 0\,.\]
 From $a_\tau(f_{n_k})\leq u(f_{n_k})$ we deduce that $\psi^\tau\leq \phi$, and thus, $u(f)= \sup_{\tau\ge 0} a_\tau(f)\leq \phi$, which finish the proof of the lower-semicontinuity.
\par\smallskip
(iii) and (iv) can be proved reasoning as for (i) and (ii) respectively, but inverting the appropriate inequalities accordingly. (v) follows from~\eqref{eq:inefora} and~\eqref{eq:ineforb}.
\end{proof}
Next we introduce the concepts of sublinear functions and sublinear skew-product semiflows in our context.
\begin{defn}\label{sublinear}
A function $f\in\LC(\R^{2N},\R^N)$ is said to be \emph{sublinear} if
\begin{enumerate}[label=\upshape(\textbf{S}),leftmargin=20pt]
\item\label{S} for each $x,y\in\R^N$, $x$, $y\ge 0$, and  each $\lambda\in[0,1]$
\begin{equation*}
f(t,\lambda\, x,\lambda\, y)\ge \lambda f(t,x,y)\;   \quad \text{ for a.e. } t\in\R\,.
\end{equation*}
\end{enumerate}
A subset $E\subset\LC(\R^{2N},\R^N)$ is said to be \emph{sublinear} or to satisfy~\ref{S} if all its elements are sublinear.
\end{defn}
\begin{defn}\label{def:skew-prod-sublinear}
The skew-product semiflow~\eqref{monosemitau}  is said to be \emph{sublinear} if for each function $f\in E$
\begin{equation}\label{defi:sub}
x_t(\cdot, f,\lambda\,\phi)\ge \lambda\, x_t(\cdot, f,\phi) \quad \text{whenever } \; t> 0,\;  \lambda\in[0,1] \text{ and } \phi\ge 0\,.
\end{equation}
A function $f\in\LC(\R^{2N},\R^N)$ is said to be a \emph{point of strong sublinearity} for $\Phi$ if
\begin{equation}\label{defi:strongsub}
x_t(\cdot, f,\lambda\,\phi)\gg \lambda\, x_t(\cdot, f,\phi) \quad \text{whenever } \; t> 1,\;  \lambda\in(0,1) \text{ and } \phi\gg 0.
\end{equation}
\end{defn}
The next technical lemma allows us to pass from~\ref{S}, formulated pointwise, to a similar condition involving continuous functions. We omit the proof, analogous to the one of Lemma~\ref{lem:f(t,a,b)->f(t,x(t),y(t))}.
\begin{lem}\label{lem:S}
If $f\in\LC(\R^{2N},\R^N)$ satisfies~\ref{S}, then for each $\lambda\in [0,1]$, any interval $I\subset\R$ and any functions $x,\,y\in \mathcal{C}(I,\R^N)$, with $x$, $y\ge0$  one has that
\begin{equation*}
f\big(t,\lambda\,x(t),\lambda\,y(t)\big)\ge \lambda \,f\big(t,x(t),y(t)\big),\qquad \text{for a.e. } t\in I.
\end{equation*}
\end{lem}
The sublinear character of the skew-product semiflow can be deduced from the sublinearity of all the functions in $E$ as shown in the next result.
\begin{prop}\label{prop:sublinear} Consider $\Phi$  one of the monotone continuous skew-product semiflows given in~\eqref{monosemitau} and assume that $E$ satisfies property~\ref{S}. Then,~\eqref{monosemitau} preserves the positive cone and induces a monotone and sublinear skew-product semiflow
\begin{equation}\label{sublinearsemitau}
\Phi \colon \R^+\times E\times\mathcal{C}^+\to E\times\mathcal{C}^+,\quad (t,f,\phi)\mapsto  \big(f_t, x_t(\cdot,f,\phi)\big)\,.
\end{equation}
\end{prop}
\begin{proof} The case $\lambda=1$ is trivial so let us fix  $f\in E$, $\lambda\in [0,1)$ and $\phi\gg0$. We firstly show that $x(t, f,\lambda\,\phi)\ge \lambda\, x(t, f,\phi)$ for $t\in[0,1]$.  Take
\[t_1=\sup\{\tau\in[0,1]\mid x(t,f,\phi)\ge 0\text{ for all }t\in[0,\tau]\}\,.\] It is obvious that $t_1>0$. For each $t\in [0,t_1]$ denote by  $v(t)=\lambda \,x(t,f,\phi)$ and $w(t)= x(t, f , \lambda \,\phi)$ and notice that
$w'(t)= g(t,w(t))$
where $g:\R\times \R^N \to \R^N$, $(t,y)\mapsto f(t,y,\lambda\,\phi(t-1))$.
Hence, from~Lemma~\ref{lem:S}
\begin{align*}
v'(t)-g(t,v(t))&= \lambda \,f(t,x(t,f,\phi),\phi(t-1))-g(t,v(t))\leq 0= w'(t)-g(t,w(t))
\end{align*}
a.e. in $[0,t_1]$, which together with  $v(0)=\lambda\,\phi(0)=w(0)$ and Theorem 2 of~\cite{paper:Walter} provides $v(t)\leq w(t)$ for $t\in[0,t_1]$. Next, from Lemma~\ref{lema:strict} and $\phi(0)\gg \lambda\, \phi(0)$  we deduce that $x(t_1,f,\phi)\gg x(t_1,f,\lambda\,\phi)$, which along with $x(t_1,f,\lambda\,\phi)\geq\lambda\, x(t_1,f,\phi)\ge0$ shows that it must be $t_1=1$, and we obtain the claimed inequality. In a recursive way we show the case $t>1$, and the sublinearity of the semiflow follows. Finally, notice that from the monotonicity, if $\phi\geq 0$ we deduce that $x_t(\cdot,f,\phi)\geq x_t(\cdot,f,0)$ and from~~\eqref{defi:sub} with $\lambda=0$ the invariance of the positive cone $\mathcal{C}^+$ is obtained.
\end{proof}
As a final result of this section, we show that under the assumptions of Theorem~\ref{thm:equilibria} with the addition of sublinearity for the monotone continuous skew-product semiflows~\eqref{sublinearsemitau}, further important dynamic information can be drawn upon the system. Indeed, one can single out two invariant subsets of $E$ where either the semicontinuous equilibria provided by Theorem~\ref{thm:equilibria}  are in fact continuous and coincide, or where they determine the forward long-term behavior of the solutions. In order to state and prove the result, let us firstly recall the definition of part metric.
\begin{defn}[Part metric]
Consider the  equivalence relation on $\mathcal{C}^+$ defined by $x\sim y$ if and only if there is $\alpha>0$ such that $\alpha^{-1}x\le y\le \alpha x$. The classes of equivalence in  $\mathcal{C}^+$ are called the \emph{parts of} $\mathcal{C}^+$.\par
If $C$ is a part of $\mathcal{C}^+$, then
\[
p(x,y):=\inf\{\log\alpha\mid \alpha^{-1}x\le y\le \alpha x\},\quad x,y\in C,
\]
defines a metric on $C$ called the \emph{part metric} of $C$.
\end{defn}
It is easy to check that $\Int \mathcal{C}^+=\{\phi\in \mathcal{C}\mid \phi\gg 0\}$ is a part of $\mathcal{C}^+$.
\begin{thm}\label{theoremE+E-}
  Let $\Phi$ be one of the monotone and sublinear continuous skew-product semiflows~\eqref{sublinearsemitau} induced by~\eqref{monosemitau} in the sublinear case. Let $u,\,v\colon E\to\Int\mathcal{C}^+$ be the semicontinuous equilibria provided by {\rm Theorem~\ref{thm:equilibria}} from the semiequilibria $a$, $b\colon E\to\Int\mathcal{C}^+$,  and consider the sets
\begin{align*}
E_-&:=\left\{f\in E\;\left| \begin{array}{l}
 \text{there is a sequence } \,t_n\downarrow-\infty  \text{ and } \\
 \text{a point of strong sublinearity } g \in E
 \end{array}\right. \; \text{ with } \lim_{n\to\infty}f_{t_n}=g\right\}\,,\\
E_+&:=\left\{f\in E\;\left| \begin{array}{l}
 \text{there is a sequence } \,t_n\uparrow +\infty  \text{ and } \\
\text{a point of strong sublinearity } g \in E
 \end{array}\right. \;\text{ with } \lim_{n\to\infty}f_{t_n}=g\right\}\,.
\end{align*}
Then,
\begin{itemize}[leftmargin=25pt,itemsep=2pt]
\item[\rm (i)] $E_-$ and $E_+$ are invariant.
\item[\rm (ii)] Each function $f\in E_-$ is a continuity point for $u$ and $v$ and  $u(f)=v(f)$. In particular, $u$ and $v$ are continuous in $E_-$.
\item[\rm (iii)] For each $f\in E_+$ and  $\phi\in\Int \mathcal{C}^+$
\begin{equation}\label{eq:limitp}
\lim_{t\to\infty}p\big(u(f_t),x_t(\cdot,f,\phi)\big)=\lim_{t\to\infty}p\big(v(f_t),x_t(\cdot,f,\phi)\big)=0.
\end{equation}
In particular, if $(t_n)_\nin$ is such that $t_n\uparrow\infty$ and $\lim_{\nti}f_{t_n}=g\in E$, then
\begin{equation}\label{eq:limitnorm}
\lim_{\nti}\|u(f_{t_n})-x_{t_n}(\cdot,f,\phi)\|_\mathcal{C}=\lim_{\nti}\|v(f_{t_n})-x_{t_n}(\cdot,f,\phi)\|_\mathcal{C}=0.
\end{equation}
\end{itemize}
\end{thm}
\begin{proof}
(i) Let $f\in E_-$ and $s\in\R$. If we consider the sequence $(t_n-s)_{\nin}$, it is immediate to check that $\lim_{n\to\infty}(t_n-s)=-\infty$ and  $\lim_{n\to\infty} (f_s)_{t_n-s}=g$,  so that $f_s\in E_-$.\par\smallskip
(ii) Let $f\in E_-$. Then, there exists a sequence $t_n\downarrow -\infty$, which can be assumed to satisfy $t_n<0$ and $t_{n-1}-t_n>2$ for each $n\in\N$,  and a point of strong sublinearity $g\in E$ such that $\lim_{n\to\infty}f_{t_n}=g$. First we check that $u(f)=v(f)$. Since $u$ is a lower-semicontinuos equilibrium and $v$ an upper-semicontinuos equilibrium, up to a subsequence, there are $\widetilde u$ and $\widetilde v\in E$ such that
\begin{equation*}
0\ll u(g)\leq \widetilde u=\lim_{n\to\infty}u\big(f_{t_n}\big)\leq \lim_{n \to\infty}v\big(f_{t_n}\big)=\widetilde v\leq v(g)\,.
\end{equation*}
First notice that since $u$ and $v$ are equilibria, and the part metric is decreasing along the trajectories because of the sublinearity of the semiflow
(see~\cite[Lemma 4.2.1(i)]{book:chue}), we deduce that
\begin{equation}\label{eq:decrease}
p(u(f_t),v(f_t))\ge p(u(f_s),v(f_s)) \quad \text{ whenever } \;t \leq s\,.
\end{equation}
Next we claim that $p(\widetilde u,\widetilde v)=0$. Otherwise, since $g$ is a point of  strong sublinearity (see~\eqref{defi:strongsub}), the contracting property under the part metric for the trajectory (see~\cite[Lemma 4.2.1(ii)]{book:chue}), the continuity of the semiflow, the inequalities $t_{n}+2 <t_{n-1}$ for each $n\in\N$ and~\eqref{eq:decrease} provide
\begin{align*}
p(\widetilde u,\widetilde v)& >p(x_2(\cdot,g,\widetilde u)), x_2(\cdot,g,\widetilde v))=\lim_{n\to\infty} p(x_2(\cdot,f_{t_n},u(f_{t_n})), x_2(\cdot,f_{t_n},v(f_{t_n})))\\ & = \lim_{n\to\infty} p(u(f_{t_n+2}),v(f_{t_n+2}))\ge \lim_{n\to\infty} p(u(f_{t_{n-1}}),v(f_{t_{n-1}}))=p(\widetilde u,\widetilde v)\,,
\end{align*}
a contradiction, and $p(\widetilde u,\widetilde v)=0$, as claimed. Finally, again from~\eqref{eq:decrease} and $t_n<0$ we deduce that
$0\leq p(u(f),v(f))\leq p(u(f_{t_n}),v(f_{t_n}))$, which as $n$ goes to $\infty$ yields $0\leq p(u(f),v(f))\leq p(\widetilde u,\widetilde v)=0$, that is, $u(f)=v(f)$, as stated.\par\smallskip
Now we check that each $f\in E_-$ is a continuity point for $u$ and $v$. Again, from the lower- and upper-semicontinuity of the equilibria  and Theorem 4.4(v), we deduce that for each sequence $(f_n)_{n\in\N}$  converging to $f$, there is a subsequence $(f_{n_k})_{k\in\N}$ such that
\[
u(f)\leq\lim_{k\to\infty}u(f_{n_k})\leq \lim_{k\to\infty}v(f_{n_k})\leq v(f)\,.
\]
Thus, from $u(f)=v(f)$ we conclude that $\displaystyle \lim_{k\to\infty}u(f_{n_k})=\lim_{k\to\infty}v(f_{n_k})=u(f)=v(f)$, and the continuity of $u$ and $v$ in $E_-$ is obtained.\par\smallskip
(iii) We will show the results for $u$ because the corresponding ones for $v$ are analogous. Let $f\in E_+$ and $\phi\gg 0$. As in~\eqref{eq:decrease}, from the sublinear character of the semiflow, the function $P(t)=p(u(f_t), x_t(\cdot,f,\phi))$ is positive and decreasing in $t>0$, that is,  $P(t)\geq P(s)\geq 0$ whenever $0<t\leq s$, and hence there exists the limit as $t\uparrow\infty$. We check that it is 0 as claimed.\par\smallskip
 Since $f\in E_+$, there is a sequence $t_n\uparrow \infty$ which can be assumed to satisfy $t_n>0$ and $t_{n+1}-t_n>2$ for each $n\in\N$,  and a point of strong sublinearity  $g\in E$ such that $\lim_{n\to\infty}f_{t_n}=g$. Again, from the lower-semicontinuity of $u$ there is a subsequence, we will take the whole one, and $\widetilde u\in E$ satisfying
 $0\ll u(g)\leq \widetilde u=\lim_{n\to\infty}u\big(f_{t_n}\big)$. \par\smallskip
 We assume that $p(u(f),\phi)=P(0)>0$, otherwise $P$ would vanish on $\R^+$ and the claim is trivially true. Then, there is an $\alpha>1$ such that $\alpha^{-1}u(f)\leq \phi\leq \alpha\,u(f)$, and the monotonicity of the semiflow provides
 \begin{equation}\label{eq:desimonot}
  x_t(\cdot,f,\alpha^{-1}u(f))\leq x_t(\cdot,f,\phi)\leq x_t(\cdot,f,\alpha\,u(f))\,.
 \end{equation}
 Moreover, from the sublinearity of the semiflow, since $\alpha^{-1}<1$ we deduce that
 \begin{align*}&\alpha^{-1} u(f_t)=\alpha^{-1}x_t(\cdot,f,u(f))\leq x_t(\cdot,f,\alpha^{-1} u(f))\quad \text{ and }  \\
 & x_t(\cdot,f,\alpha\,u(f))\leq \alpha\,x_t(\cdot,f,u(f))=\alpha\,u(f_t)\,,
 \end{align*}
 which, together with~\eqref{eq:desimonot} and evaluating at $t_n$, show that $\{x_{t_n}(\cdot,f,\phi)\mid \nin\}$ is bounded because
 \[ \alpha^{-1} u(f_{t_n})\leq x_{t_n}(\cdot, f,\phi)\leq \alpha\,u(f_{t_n})\]
 and $u(f_{t_n})\to \widetilde u$ as $n\uparrow \infty$. We omit the proof of the equicontinuity because follows the same arguments of Theorem~\ref{thm:equilibria} for proving that  $\{u(f_n)\mid \nin\}$ was equicontinuous. Consequently,  $\{x_{t_n}(\cdot,f,\phi)\mid \nin\}$ is a relatively compact set and there is a convergent subsequence. For simplicity of notation we will assume that the sequence itself converges, and  let denote the limit by $\widetilde x$. We claim that $p(\widetilde u,\widetilde x)=0$.  Otherwise, the same arguments of (ii) together with $t_n+2 < t_{n+1}$ provide now the chain of inequalities
\begin{align*}
p(\widetilde u,\widetilde x) &> p(x_2(\cdot,g,\widetilde u), x_2(\cdot,g,\widetilde x)) =\lim_{n\to\infty} p(x_2(\cdot,f_{t_n},u(f_{t_n})),x_2(\cdot,f_{t_n},x_{t_n}(\cdot,f,\phi))\\
& =\lim_{n\to\infty} P(t_n+2)\geq \lim_{n\to\infty} P(t_{n+1})=p(\widetilde u,\widetilde x)\,,
\end{align*}
and hence, $p(\widetilde u,\widetilde x)=0$ and $\displaystyle \lim_{t\to\infty} P(t)=0$, as stated.
Finally,~\eqref{eq:limitnorm} follows from~\eqref{eq:limitp}, the relation (see Krause and Nussbaum~\cite[Lemma 2.3(ii)]{paper:KrNu})
\[\|x-y\|\leq\left(2\,e^{p(x,y)}-e^{-p(x,y)}-1\right)\,\min(\|x\|,\|y\|)\,,\]
and the boundedness of $\{u(f_{t_n})\mid n\in\N\}$ and $\{x_{t_n}(\cdot,f,\phi)\mid \nin\}$.
\end{proof}
\section{Some applications}\label{sec:models}   In this section, we will show the importance of the applications of our
theory to the study of nonautonomous Carath\'{e}odory ordinary and delay cooperative
systems of equations arising in several applied sciences.
\subsection{Scalar model for population dynamics}
The use of scalar differential equations with constant delay to model the dynamics of a population is extensive and effective (see Brauer and Ch\'{a}vez \cite{book:BC} and Smith \cite{book:S}).  We will study a model in population biology given by scalar delay Carath\'{e}odory equations of the form
\begin{equation}\label{eq:main}
x'(t)=-\alpha(t)\,x(t)+h(t,x(t-1))\,,
\end{equation}
which will be compared with the linear Carath\'eodory ones
\begin{equation}\label{eq:lineal}
y'(t)=-\alpha(t)\,y(t)+\beta(t)\,y(t-1)\,+\gamma(t)\,.
\end{equation}
Equations like \eqref{eq:main} include for example Nicholson's model, Mackey-Glass's model and similar, which have been studied in depth in mathematical biology. First we state the assumptions to be considered for the families of equations \eqref{eq:main} and \eqref{eq:lineal}.
\begin{enumerate}[label=\upshape(\textbf{A\arabic*}),series=scalar,leftmargin=27pt,itemsep=2pt]
\item\label{A1}
$E_1$ is a closed invariant and bounded subset of functions $\alpha\in L^1_{loc}$  such that $\alpha(t)\ge 0$ for a.e. $t\in\R$ and the null function $\alpha=0$ does not belong to~$E_1$. Let $D$ be a countable dense subset of $\R$ and consider  the subset of $\LC$  given~by
\[
E_2=\left\{h\in\LC\,\Bigg|\,\begin{array}{l}
h \text{ satisfies~\ref{K2}, } h_0(t):=h(t,0)\in E_1 \\[.1cm]
\text{and } h(t,y)=h(t,0)\text{ for }y\le0
\end{array}\right\}.
\]
Notice that $E_2$ is invariant and closed with respect to the topology $\T_D$.
\item\label{A3} $E$ is a subset of
\[
\left\{f=(h,\alpha,\beta,\gamma)\,\Bigg|\,\begin{array}{l}
h\in E_2,\; \alpha,\beta,\gamma\in E_1\text{ and } \\[.1cm]
h(t,y)\leq \beta(t)\,y+\gamma(t),\, \forall y\in\R, \text{ a.e.} \, t\in\R
\end{array}\right\}
\]
 which is invariant and closed for the product topology.
\item\label{A4} There are positive constants $K$, $\delta>0$ such that for each $f=(h,\alpha,\beta,\gamma)\in E$
\[\|T(t,f)\|\leq K\,e^{-\delta \,t} \quad \text{ for each }t>0\,,\]
where $T(t,f)$ is the evolution operator on $C([-1,0],\R)$ for~\eqref{eq:linealhomo}, that is, $T(t,f)\,\phi=x_t(\cdot,f,\phi)$ is the unique solution of the linear equation
\begin{equation}\label{eq:linealhomo}
z'(t)=-\alpha(t)\,z(t)+\beta(t)\,z(t-1)
\end{equation}  with initial data $\phi\in C([-1,0],\R)$.
\end{enumerate}
Under these assumptions, the functions defining equations~\eqref{eq:main} and~\eqref{eq:lineal}, that is,  $g(t,x,y)=\alpha(t)\,x+h(t,y)$ and $\widehat g(t,x,y)=\alpha(t)\,x+\beta(t)y+\gamma(t)$,  satisfy conditions~\ref{Kx},~\ref{Ky} and~\ref{Lx}.  Therefore,  from Theorem~\ref{thm:monoton<=>Kamke} and Theorem~\ref{thm:K_xK_y+lipschitz=>cont skew-prod} we deduce that the skew-product semiflows:
\begin{equation}\label{semiflowh}
\Phi\colon \R^+\times E\times\mathcal{C}^+\to E\times\mathcal{C}^+,\quad (t,f,\phi)\mapsto  \big(f_t, x_t(\cdot,f,\phi)\big)\,,
\end{equation}
\begin{equation}\label{semiflowlinear}
\Psi \colon \R^+\times E\times\mathcal{C}^+\to E\times\mathcal{C}^+,\quad (t,f,\phi)\mapsto  \big(f_t, y_t(\cdot,f,\phi)\big)\,,
\end{equation}
 are monotone and continuous for the above product topology.
\begin{rmk}\label{fundamental} Let $f=(h,\alpha,\beta,\gamma)\in E$. We denote by $U_f(t,s)$ the so called \emph{fundamental matrix}  of~\eqref{eq:linealhomo} (scalar in this case, see~\cite[Theorem 2.1]{book:HALE}), that is,
\begin{align}\nonumber
&\frac{d}{dt} U_f(t,s)=-\alpha(t) \,U_f(t,s)+\beta(t)\, U_f(t-1,s),  \;\text{ if } t\geq s  \text{ and a.e. in } s \text{ and } t\,. \\[.1cm]
& U_f(s,s)=1 \text{ and }  U(t,s)=0 \text{ for } s-1\leq t<s\,.\label{initialc}
\end{align}
Moreover, it is also assumed that $U_f(t,s)=0$ whenever $t<s$.
Notice that for each $f=(h,\alpha,\beta,\gamma)\in E$ and assumption~\ref{A4} we have
\begin{equation}\label{eq:inefun}
|U_f(t,s)|\leq K\, e^{-\delta\,(t-s)} \quad \text{ whenever }t\geq s\,.
\end{equation}
 The reason is that if we change the initial condition~\eqref{initialc} to the constant function 1 on $[s-1,s]$,  and we denote this solution by $z(t,f,s,1)$, it is easy to check that \[U_f(t,s)\leq z(t,f,s,1)=z(t-s,f_s,1)\,,\]
 and~\ref{A4} proves the claim.
\end{rmk}
\begin{prop}\label{equilibrios}
Under assumptions~\ref{A1}--\ref{A4} and the notation of\/ \rm{Remark~\ref{fundamental}},
\begin{itemize}[leftmargin=20pt]
\item[\rm(i)] the function
$\widetilde b\colon E\to\R\,,$ $f\mapsto \widetilde b(f)=\int_{-\infty}^0 U_f(0,s)\,\gamma(s)\,ds$
is well defined and continuous,
\item[\rm(ii)] the  function $b\colon E\to \mathcal{C}^+\,,$ $f\mapsto b(f)$,  defined as $b(f)(u)=\widetilde b(f_u)$ for each $u\in [-1,0]$,  is a continuous equilibrium for~\eqref{semiflowlinear} and a continuous super-equilibrium for~\eqref{semiflowh},
\item[\rm(iii)] the function
$\widetilde a\colon E\to\R\,,$ $f\mapsto \widetilde a(f)=\int_{-\infty}^0 \exp\big(-\int_s^0 \alpha(r)\,dr\big)\,h_0(s)\,ds$, with $h_0$ defined in~\ref{A1},
is well defined and continuous,
\item[\rm(iv)] the  function $a\colon E\to \mathcal{C}^+\,,$ $f\mapsto a(f)$,  defined as $a(f)(u)=\widetilde a(f_u)$ for each $u\in [-1,0]$,  is a continuous sub-equilibrium for~\eqref{semiflowh}, and
\end{itemize}
they satisfy $0\ll a(f)\leq b(f)$ for each $f\in E$.
\end{prop}
\begin{proof}
(i) Notice that~\ref{A1} implies that the set $\{\gamma_t\mid \gamma\in E_1, t\in\R\}$ is $L^1_{loc}$-bounded, so there is a constant $C\geq 0$ such that \begin{equation}\label{eq:boundL1loc}
\int_t^{t+1} \gamma(s)\,ds\leq C\quad \text{ for each } t\in\R \text{ and } \gamma\in E_1\,.
\end{equation}
 From this together with~\eqref{eq:inefun}  we deduce that for each $f=(h,\alpha,\beta,\gamma)\in E$
\[
\left| \int_{-\infty}^0 U_f(0,s)\,\gamma(s)\,ds\right|\leq \sum_{j=1}^{\infty} \int_{-j}^{-j+1} \!\!K\,e^{\delta\,s}\gamma(s)\ ds\leq
\sum_{j=1}^{\infty} K C e^{(-j+1)\,\delta}= K C \frac{e^{\delta}}{e^\delta-1}\,,
\]
so that  the integral is well defined, and since the bound is independent of $f$,
we deduce that given $\ep>0$ there is a $\tau_\ep>0$ such that
\begin{equation}\label{eq:inetau}
\left| \int_{-\infty}^{-\tau_\ep} U_f(0,s)\,\gamma(s)\,ds\right| <\ep \quad \text{ for each } f=(h,\alpha,\beta,\gamma)\in E\,.
\end{equation}
Before proving the continuity of $\widetilde b$, we will study some properties of $U_f(0,s)$. From~\eqref{initialc}, if $s\leq t\leq s+1$ we deduce that $U_f(t,s)$ is the solution of the Carath\'{e}odory ordinary differential problem
$z'(t)=-\alpha(t)\,z(t)$, $z(s)=1$, i.e.
\begin{equation}\label{eq:U s< t+1}
U_f(t,s)=\exp\left(-\int_s^t \alpha(r)\,dr\right) \quad\text{ whenever } s\leq t\leq s+1\,.
\end{equation}
As in Remark~\ref{fundamental}, we denote by $z(t,f,\phi)$ the solution of
\[\begin{cases} z'(t)=-\alpha(t)\,z(t)+\beta(t)\,z(t-1)\,,\\
z(t)=\phi(t),\quad t\in[-1, 0]\,,
\end{cases} \text{for } f=(h,\alpha,\beta,\gamma)\in E, \text{ and } \phi\in \mathcal C([-1,0])\,.
\]
Since the cocycle property for $z_t(\cdot,f, \phi)$ can be also applied to the non-continuous initial data $\phi_0\colon [-1,0]\to\R$,
defined as $\phi_0(t)=0$ if $t\in[-1,0)$ and $\phi_0(0)=1$, we deduce that if $t\geq s+1$
\begin{equation}\label{eq:s> t+1}
U_f(t,s)=z(t-s,f_s,\phi_0)= z(t-s-1, f_{s+1},z_1(\cdot, f_s,\phi_0))\,,
\end{equation}
where now, from the definition of $\phi_0$, we have that $z(1+u, f_s,\phi_0)$  is the solution of the Carath\'{e}odory ordinary differential equation  $z'(t)=-\alpha(t+s) \,z(t)$ with initial data $z(0)=1$, and hence,
\begin{equation}\label{phins}
\phi(f_s)(\tau):=z_1(\tau, f_s,\phi_0)=\exp\left(-\int_0^{1+\tau}\alpha(s+r)\,dr\right), \tau\in[-1, 0]\,.
\end{equation}
Therefore, from~\eqref{eq:U s< t+1},~\eqref{eq:s> t+1} and~\eqref{phins} we deduce that
\begin{equation}\label{eq:Ufn(0,s)}
U_f(0,s)=\begin{cases}  \exp\big(-\int_s^0 \alpha(r)\,dr\big) & \text{if } \;s\in[-1,0]\,, \\[.1cm]
                               z(-s-1, f_{s+1},\phi(f_s)) & \text{if }\; s\leq -1\,.
\end{cases}
\end{equation}
From Theorem~\ref{thm:K_xK_y+lipschitz=>cont skew-prod},  the map
$\,\R^-\times E\times \mathcal C([-1,0])\to\R, \;(s,g,\phi)\mapsto z(-s,g,\phi))\,$ is continuous,
which, together with~\eqref{eq:Ufn(0,s)} and the composition with some continuous functions, show the continuity of
\begin{equation}\label{Uf0s}
\R^-\times E\to \R,\; (s,f)\mapsto U_f(0,s)\,.
\end{equation}
To check the continuity of $\widetilde b$, we consider a sequence $(f_n)_{\nin}=(h_n,\alpha_n,\beta_n,\gamma_n)_\nin$ on $E$ converging to some $f=(h,\alpha,\beta,\gamma)\in E$ for the product topology and we will  check that $\widetilde b(f_n)$ tends to $\widetilde b(f)$ as $n\uparrow \infty$. Given $\ep >0$,  we consider $\tau_\ep$  satisfying~\eqref{eq:inetau} and we denote by $I_\ep=\sup_{\gamma\in E_1}\int_{-\tau_\ep}^0 \gamma(s)\,ds$. From the definition of $E_1$ given on~\ref{A1} we deduce that $0<I_\ep <\infty$. Moreover,  notice that the restriction of~\eqref{Uf0s} to the compact subset $[-\tau_\ep,0]\times \big(\{f_n\mid \nin\}\cup\{f\}\big)\subset \R^-\times E$ provides
\[\lim_{n\to\infty}U_{f_n}(0,s)= U_f(0,s) \quad \text{ uniformly for } s\in[-\tau_\ep, 0]\,.\]
In addition, $\lim_{n\to\infty} \gamma_n=\gamma$ in $L^1_{loc}$, so that there is an $n_0$ such that for each $n\ge n_0$
\[ \sup_{s\in[-\tau_\ep,0]} |U_{f_n}(0,s)-U_f(0,s)| < \frac{\ep}{I_\ep} \quad \text{ and }\; \int_{-\tau_\ep}^0 |\gamma_n(s)-\gamma(s)|<\ep\,.\]
As a consequence, together with~\eqref{eq:inefun},~\eqref{eq:boundL1loc},~\eqref{eq:inetau} and
\[U_{f_n}(0,s)\,\gamma_n(s)- U_f(0,s)\,\gamma (s)= (U_{f_n}(0,s)-U_f(0,s))\,\gamma_n(s) + U_f(0,s)\,(\gamma_n(s)-\gamma(s))\,,\]
we obtain that for each $n\ge n_0$
\begin{align*}
|\widetilde b(f_n)-\widetilde b(f)|&\leq 2\,\ep +  \int_{-\tau_\ep}^0 \left|U_{f_n}(0,s)\,\gamma_n(s)- U_f(0,s)\,\gamma (s)\right|\,ds\\
   & \leq 2\,\ep +\frac{\ep}{I_\ep} \int_{-\tau_\ep}^0 |\gamma_n(s)|\, ds+ \ep\,\sup_{s\in\R^-}|U_f(0,s)|\leq \ep\, (3 + K)\,,
\end{align*}
which proves our claim, and finishes the proof of (i).\par\smallskip
\noindent (ii) First notice that, from the variation of constant formula, the solution~of
\[\begin{cases} y'(t)=-\alpha(t)\,y(t)+\beta(t)\,y(t-1)+\gamma(t)\,,\\
y(t)=0,\quad t\in[s-1, s]\,,
\end{cases}
\]
for $f=(h,\alpha,\beta,\gamma)\in E$  is given by  $\int_s^t U_f(t,s)\,\gamma(s)\,ds$  for each  $t\geq s$,
and, as a consequence of~\eqref{eq:inefun}, it can be shown that
\begin{equation}\label{boundedsolut}
y(t,f,b(f)))=\int_{-\infty}^t U_f(t,s)\,\gamma(s)\,ds\,
\end{equation}
is a bounded and globally defined solution of~\eqref{eq:lineal} such that $y(s,f,b(f))=b(f)(s)$ for each $s\in[-1,0]$.
\par
Next, as in~\eqref{eq:s> t+1}, from $U_f(t,s)=z(t-s,f_s,\phi_0)$ we deduce that
\[U_{f_t}(0,s-t)=z(t-s,(f_t)_{s-t},\phi_0)=z(t-s,f_s,\phi_0)=U_f(t,s)\,.\]
Therefore,
\[\widetilde b(f_t)=\int_{-\infty}^0 U_{f_t}(0,s)\,\gamma_t(s)\,ds=\int_{-\infty}^t U_{f_t}(0,s-t)\,\gamma(s)\,ds=\int_{-\infty}^t U_f(t,s)\,\gamma(s)\,ds\,,\]
 and from~\eqref{boundedsolut} it follows that $\widetilde b(f_t)=y(t,f,b(f))$. Thus, $b(f)$ is a continuous equilibrium for~\eqref{semiflowlinear} because $b(f_t)(s)=\widetilde b(f_{t+s})=y(t+s,f,b(f))$ for $s\in[-1,0]$, i.e.  $b(f_t)=y_t(\cdot,f, b(f)))$, as stated. In addition, from~\eqref{eq:inefun} it follows that this equilibrium $b(f)$ is globally exponentially stable for $\Psi$ in \eqref{semiflowlinear}.  Finally notice that from~\ref{A3}, again a comparison result provides  $x_t(\cdot,f,b(f))\leq y_t(\cdot,f,b(f))=b(f_t)$ for $t\ge0$, and we conclude that $b(f)$ is a continuous super-equilibrium for~\eqref{semiflowh}, which finishes the proof of (ii).\par\smallskip
 We will omit the proofs of (iii) and (iv) because follow similar arguments to those of (i) and (ii). We want only to remark the following differences. Notice  that $Z_f(t,s):=\exp\big(-\int_s^t\alpha(u)\,du\big)$ is the fundamental matrix of the Carath\'{e}odory ordinary differential equation~$z'(t)=-\alpha(t)\,z(t)$ with $Z_f(s,s)=1$,  $U_f(t,s)$ is the fundamental matrix  of~\eqref{eq:linealhomo}, and  from~\ref{A3} and~\eqref{eq:inefun} we also deduce that $Z_f(t,s)\le U_f(t,s)\leq K\, e^{-\delta\,(t-s)}$ whenever $s\leq t$. Moreover, in (i) and (ii)  we were comparing the solutions of equations~\eqref{eq:main} and~\eqref{eq:lineal}, i.e. $x(t,f,\phi)\leq y(t,f,\phi)$, so that it is said that~\eqref{eq:lineal} is a \emph{majorant} of~\eqref{eq:main},  and now again from~\ref{A1} and~\ref{A3}, the Carath\'{e}odory ordinary differential equation $y'(t)=-\alpha(t)\,y(t)+h_0(t)$ is a \emph{minorant} of~\eqref{eq:main}. This is the reason why now $a(f)$, which is a globally exponentially stable continuous equilibrium for $\Psi$ in \eqref{semiflowlinear}, is a continuous sub-equilibrium for~\eqref{semiflowh}.\par\smallskip
 Finally, $a(f)\leq b(f)$ for each $f\in E$ follows from the above comparison of solutions, and $0\ll a(f)$ for each $f\in E$ is a consequence of the definition of $a(f)$ and the belonging of  $h_0$ to the set $E_1$, which is closed, invariant and does not contain the null function, from which it can be shown that
the set  $\{r\in[-\infty, u]\mid h_0(r)>0\}$  have positive measure for each $u\in[-1,0]$.
\end{proof}
As a consequence, from Theorem~\ref{thm:equilibria} we deduce the existence of two equilibria for the skew-product semiflow~\eqref{semiflowh} induced by the family of equations~\eqref{eq:main}, one is lower-semicontinuous  $u\colon E\to\mathcal{C}^+,\;f\mapsto u(f)$, the other is  upper-semicontinuous  $v\colon E\to\mathcal{C}^+,\;f\mapsto v(f)$, and they satisfy $0\ll u(f)\leq v(f)$. It can be shown that they define the top and lower covers of a \emph{pullback attractor} for the evolution processes induced by~\eqref{eq:main} (we refer the reader to~Kloeden and Rasmussen~\cite{book:KR} for these concepts).\par\smallskip
In addition, we assume the following condition implying  the sublinearity of the skew-product semiflow:
\begin{enumerate}[resume*=scalar]\setlength\itemsep{2pt}
\item\label{A5} for each function $f=(h,\alpha,\beta,\gamma)\in E$, the function $h$ is sublinear,  that is, for each $y\in\R^+$ and $\lambda\in[0,1]$,
$\;h(t,\lambda\, y)\ge \lambda\, h(t,y)$ for a.e. $t\in\R$.
\end{enumerate}
From this, we deduce that $g(t,x,y)=-\alpha(t)\,x+h(t,y)$  is also sublinear, i.e. satisfies~\ref{S}, and hence the skew-product semiflow~\eqref{semiflowh} is sublinear, as shown in Proposition~\ref{prop:sublinear}. Next lemma characterizes the points of strong sublinearity for the skew-product semiflow~\eqref{semiflowh} (see~\eqref{defi:strongsub} for the definition).
\begin{lem} Consider $f=(h,\alpha,\beta,\gamma)\in E$ and assume that the scalar function $h$  satisfies the following property of strong sublinearity:
\begin{itemize}[leftmargin=18pt]
\item for each $\delta>0$ the set of points  $t\in(0,\delta)$ such that
\[\quad h(t,\lambda\, y)> \lambda\, h(t,y)\;   \quad \text{for each } y>0 \text{ and }\lambda\in(0,1)\]
has positive Lebesgue measure.
\end{itemize}
Then, $f$ is a point of strong sublinearity for the skew-product semiflow~\eqref{semiflowh}.
\end{lem}
\begin{proof} We will just provide a sketch of the proof. First, notice that the function $g(t,x,y)=-\alpha(t)\,x+ h(t, y)$  satisfies the same property of strong sublinearity on $x$ and $y$. From this, reasoning as in Proposition~\ref{prop:sublinear} to transform the delay equation~\eqref{eq:main} into an ordinary one in $[0,1]$,  we check condition (ii) of~\cite[Theorem 4]{paper:Walter} to deduce that $x(t,f,\lambda \,\phi) > \lambda\,x(t,f, \phi)$ for each $t\in(0,1]$, $\lambda\in(0,1)$ and $\phi\gg 0$.
In a recursive way we can check that this also holds for $t>0$, which means that $x_t(\cdot,f,\lambda \,\phi)\gg \lambda \,x_t(\cdot,f,\phi)$ for $t>1$, as claimed.
\end{proof}
As a consequence, under assumptions~\ref{A1}--\ref{A5}, we can apply the conclusions of Theorem~\ref{theoremE+E-} to the monotone, sublinear and continuous skew-product semiflow~\eqref{semiflowh},  characterizing the points of strong sublinearity on $E_-$ and $E_+$, as in the previous lemma. Thus, the existence of a unique continuous equilibrium whose graph coincides with the pullback attractor of the equation is shown.
\subsection{Carath\'{e}odory non-autonomous cyclic feedback system} We finish with an application of our previous results to the mathematical model of biochemical feedback in protein synthesis given by the system of Carath\'{e}odory differential equations
\begin{equation}\label{bioche}
\begin{split}
x'_1(t)& =h(t,x_m(t))-\alpha_1(t)\,x_1(t)\,,\\
x'_i(t) & =x_{i-1}(t)-\alpha_i(t)\,x_i(t)\,,\qquad \text{ for } 2\leq i\leq m\,.
\end{split}
\end{equation}
The system~\eqref{bioche} expresses a model for a biochemical
control circuit in which each of the $x_j$ represents the
concentration of an enzyme; hence $x_j\ge 0$ for $j=1,\ldots,m$.
The autonomous ordinary case was firstly introduced by
Selgrade~\cite{paper:selg}. Different extensions to the
periodic and the autonomous functional cases
are explored in Smith~\cite{book:smith1995}, Krause and Ranft~\cite{paper:krra}, Smith and
Thieme~\cite{paper:smiththieme1}, and references therein. Chueshov~\cite{book:chue}
analyzes the random case and Novo {\it et al.}~\cite{paper:noos} the ordinary deterministic non-autonomous case. The case of finite-delay was considered, for the concave case, in Novo {\it et al.}~\cite{paper:nno2005}, and, for the sublinear case, in Novo and Obaya~\cite{book:noob}.\par
We state the assumptions to be considered for this problem, concerning with the Carath\'{e}odory ordinary case.  Notice that~\ref{B1} coincides with~\ref{A1} but we repeat it by completeness.
\begin{enumerate}[label=\upshape(\textbf{B\arabic*}),series=system,leftmargin=27pt,itemsep=2pt]
\item\label{B1}
$E_1$ is a closed invariant and bounded subset of functions $\alpha\in L^1_{loc}$  such that $\alpha(t)\ge 0$ for a.e. $t\in\R$ and the null function $\alpha=0$ does not belong to~$E_1$. Let $D$ be a countable dense subset of $\R$ and consider  the subset of $\LC$  given~by
\[
E_2=\left\{h\in\LC\;\Bigg|\,\begin{array}{l}
h \text{ satisfies~\ref{K2}, } h_0(t):=h(t,0)\in E_1 \\[.1cm]
\text{and } h(t,y)=h(t,0)\;\text{ for }y\le0
\end{array}\right\}.
\]
Notice that $E_2$ is invariant and closed with respect to the topology $\T_D$.
\item\label{B3}
$E$ is a subset of
\[
\hspace{0.5cm}\left\{f=(h,\alpha_1,\ldots,\alpha_m,\beta,\gamma)\;\Bigg|\,\begin{array}{l}
h\in E_2,\; \alpha_1,\ldots,\alpha_m,\beta,\gamma\in E_1\text{ and } \\[.1cm]
h(t,y)\leq \beta(t)\,y+\gamma(t), \; \forall y\in\R, \text{ a.e.} \, t\in\R
\end{array}\right\}
\]
 which is invariant and closed for the product topology.
\item\label{B4}
There are constants $K$, $\delta>0$ such that for each $f=(\alpha_1,\ldots,\alpha_m\,,\beta,\gamma)\in E$
\[\|U(t,f)\|\leq K\,e^{-\delta \,t} \quad \text{ for each }t>0\,,\]
where $U(t,f)$ is the fundamental matrix solution  of the system
\begin{equation}\label{eq:linealhomo2}
\begin{split}
z'_1(t)& =\beta(t)\,z_m(t)-\alpha_1(t)\,z_1(t)\,,\\
z'_i(t) & =z_{i-1}(t)-\alpha_i(t)\,z_i(t)\,,\qquad \text{ for } 2\leq i\leq m\,,
\end{split}
\end{equation}
principal at $t=0$, that is, $U(0,f)=I_m$.
\end{enumerate}
Under these assumptions, the function defining~\eqref{bioche} satisfies conditions~\ref{K1} and~\ref{LxODE}, and from~\cite[Theorem 2]{paper:Walter} and Proposition~\ref{thm:Kx+LipschitzODEs-cont skew-prod}, the skew-product semiflow
\begin{equation}\label{semiflowhbioche}
\Phi\colon \R^+\times E\times(\R^+)^m\to E\times(\R^+)^m,\quad (t,f,x_0)\mapsto  \big(f_t, x(t,f,x_0)\big)\,,
\end{equation}
is monotone and continuous for the above product topology.
\begin{rmk}\label{fundamental2} If we denote by $U_f(t,s)$ the fundamental matrix solution for~\eqref{eq:linealhomo2} principal at $t=s$, that is, $U_f(s,s)=I_m$, from assumption~\ref{B4} we deduce that
\[
\|U_f(t,s)\|\leq K\, e^{-\delta\,(t-s)} \quad \text{ whenever }t\geq s\,,
\]
because $U_f(t,s)=U(t-s,f_s)$. Moreover,  denoting by  $Z_f(s,t)$ the fundamental matrix solution of
\begin{equation}\label{eq:linealhomo2bioche}
\begin{split}
z'_1(t)& =-\alpha_1(t)\,z_1(t)\,,\\
z'_i(t) & =z_{i-1}(t)-\alpha_i(t)\,z_i(t)\,,\qquad \text{ for } 2\leq i\leq m\,,
\end{split}
\end{equation}
and since, from~\ref{B1}, the system~\eqref{eq:linealhomo2bioche} is a minorant of~\eqref{eq:linealhomo2}, we also have
\[
\|Z_f(t,s)\|\leq K\, e^{-\delta\,(t-s)} \quad \text{ whenever }t\geq s\,.
\]
In addition, notice that if the first component of $z_0$ is positive, i.e. $(z_0)_1>0$,  the solution of~\eqref{eq:linealhomo2bioche} with this initial data $z_0$  is strongly positive,
that is, $Z_f(t,0)\,z_0\gg 0$ for each $t> 0$.  The reason is as follows. From $z'_1(t)=-\alpha_1(t)\,z_1(t)$ we deduce that
$(Z_f(t,0)\,z_0)_1= \left(\int_0^t \exp(-\alpha_1(r))\,dr\right)(z_0)_1>0$ for each $t\ge 0$\,.
 For the second component, from  $z_2'(t)=z_1(t)-\alpha_2(t)z_2(t)> -\alpha_2(t)z_2(t)$ and comparison of solutions we deduce that
$(Z_f(t,0)\,z_0)_2 > \left(\int_0^t \exp(-\alpha_2(r))\,dr\right)(z_0)_2\geq 0$ for  each  $t>0$,
and the same reasoning for the rest of the components provides  $Z_f(t,0)\,z_0\gg 0$ for each $t>0$, as claimed.
\end{rmk}
\begin{prop} Under assumptions~\ref{B1}--\ref{B4} and notation of {\rm Remark~\ref{fundamental2}}, consider the unit vector $\textbf{e}_1=(1,0,\ldots,0)^t$. Then
\begin{itemize}[leftmargin=20pt]
\item[\rm(i)] $b\colon E\to(\R^+)^m\,,$ $f\mapsto b(f)=\int_{-\infty}^0 U_f(0,s)\,\gamma(s)\,\mathbf{e}_1\,ds$
is a continuous super-equilibrium for~\eqref{semiflowhbioche},
\item[\rm(ii)] $a\colon E\to(\R^+)^m\,,$ $f\mapsto a(f)=\int_{-\infty}^0 Z_f(0,s)\,h_0(s)\,\mathbf{e}_1\,ds$, with $h_0$ defined in~\ref{B1},
is a continuous sub-equilibrium for~\eqref{semiflowhbioche},
\end{itemize}
and they satisfy $\,0\ll a(f)\leq b(f)$.
\end{prop}
\begin{proof} We omit the proof because it is similar to that of Proposition~\ref{equilibrios}. However, we want to remark the following facts. Although we  introduced the definitions of sub and super-equilibria  for skew-product semiflows induced by Carath\'{e}odory delay differential equations, they are easily adapted to Carath\'{e}odory ordinary differential equations, changing $\mathcal{C}=C([-1,0],\R^N)$ by $\R^m$ in our case.\par
We want also to notice that in this example, the system
\begin{equation*}
\begin{split}
z'_1(t)& =\beta(t)\,z_m(t)-\alpha_1(t)\,z_1(t)+\gamma(t)\,,\\
z'_i(t) & =z_{i-1}(t)-\alpha_i(t)\,z_i(t)\,,\qquad\qquad\quad \text{ for } 2\leq i\leq m
\end{split}
\end{equation*}
is a majorant for~\eqref{bioche}  in the sense explained in the proof of Proposition~\ref{equilibrios}, and
$\,z(t)=\int_{-\infty}^t U_f(0,s)\,\gamma(s)\,\textbf{e}_1\,ds\,$ is a globally defined bounded solution for it.
Analogously, $\,z(t)=\int_{-\infty}^t Z_f(0,s)\,h_0(s)\,\textbf{e}_1\,ds\,$ is a globally defined bounded solution~of
\begin{equation*}
\begin{split}
z'_1(t)& =-\alpha_1(t)\,z_1(t)+h_0(t)\,,\\
z'_i(t) & =z_{i-1}(t)-\alpha_i(t)\,z_i(t)\,,\qquad \text{ for } 2\leq i\leq m\,,
\end{split}
\end{equation*}
which is a minorant of~\eqref{bioche} because $h$ satisfies~\ref{K1}, and then $h(t,z)\ge h(t,0)=h_0(t)$ for $z\geq 0$. These two facts are in this case the main ingredients for the proof of (i), (ii) and $a(f)\leq b(f)$ for each $f\in E$. The proof of $0\ll a(f)$ for each $f\in E$ relies now on the belonging of $h_0$ to $E_1$, so that $\{s\in (-\infty,0] \mid h_0(s) > 0\}$ has positive measure, and then for all these points, the first component of $h_0(s)\,\mathbf{e}_1$ is positive and $Z_f(0,s)\,h_0(s)\,\mathbf{e}_1 =Z_{f_s}(-s,0)\,h_0(s)\,\mathbf{e}_1\gg 0$, as shown in Remark~\ref{fundamental2}.
\end{proof}
As a in the previous scalar example, from Theorem~\ref{thm:equilibria} we deduce the existence of two equilibria for the skew-product semiflow~\eqref{semiflowhbioche} induced by the family of systems~\eqref{bioche}. One of them,  $u\colon E\to\mathcal{C}^+,\;f\mapsto u(f)$ is lower-semicontinuous, the other, $v\colon E\to\mathcal{C}^+,\;f\mapsto v(f)$, is  upper-semicontinuous, and they satisfy $0\ll u(f)\leq v(f)$. They define the top and lower covers of a \emph{pullback attractor} for the skew-product semiflow~\eqref{semiflowhbioche}. \par\smallskip
We finish with the assumption of a sublinearity condition for this case, in order to apply the conclusions of Theorem~\ref{theoremE+E-}.
\begin{enumerate}[resume*=system]\setlength\itemsep{2pt}
\item\label{B5} For each $f=(h,\alpha_1,\ldots,\alpha_m,\beta,\gamma)\in E$, the function $h$ is sublinear, that is, for each $y\in\R^+$ and $\lambda\in[0,1]$, $\;h(t,\lambda\, y)\ge \lambda\, h(t,y)$ for a.e. $t\in\R$.
\end{enumerate}
From this we deduce that the function $g\colon \R\times(\R^+)^m\to (\R^+)^m$ of system~\eqref{bioche} defined as
$g_1(t,x)=h(t,x_m)-\alpha_1(t)\,x_1$ and $g_i(t,x)=x_{i-1}-\alpha_i(t)\,x_i$ for   $2\leq i\leq m$
is sublinear, and hence the skew-product semiflow~\eqref{semiflowhbioche} is sublinear.
\par\smallskip
As in~\eqref{defi:strongsub}, but now for the case of ordinary differential equations, $f$ is a \emph{point of strong sublinearity for the skew-product semiflow}~\eqref{semiflowhbioche} if
\begin{equation}\label{defi:strongsubordi}
x(t, f,\lambda\,x_0)\gg \lambda\, x(t, f,x_0) \quad \text{whenever } \; t> 0,\;  \lambda\in(0,1) \text{ and } x_0\gg 0.
\end{equation}Next lemma provides a characterization for these points of strong sublinearity.
\begin{lem} Consider $f=(h,\alpha_1,\ldots,\alpha_m,\beta,\gamma)\in E$ and assume that the scalar function $h$  satisfies the following property of strong sublinearity:
\begin{itemize}[leftmargin=18pt]
\item for each $\delta>0$ the set of points  $t\in(0,\delta)$ such that
\begin{equation}\label{strognsublinia}
 \quad h(t,\lambda\, y)> \lambda\, h(t,y)\;   \quad \text{for each } y>0 \text{ and }\lambda\in(0,1)
\end{equation}
has positive Lebesgue measure.
\end{itemize}
Then, $f$ is a point of strong sublinearity for the skew-product semiflow~\eqref{semiflowhbioche}.
\end{lem}
\begin{proof} Take $\lambda\in (0,1)$, $x_0\gg 0$, $v(t)=\lambda \,x(t,f,x_0)$ and $w(t)= x(t, f , \lambda \,x_0)$. Since the skew-product semiflow~\eqref{semiflowhbioche} is sublinear we know that $v(t)\leq w(t)$ for each $t\geq 0$. Moreover,  consider the scalar linear equation
\begin{equation*}
y'(t)=l_1(t,y(t))=h(t,\lambda\, x_m(t,f,x_0))-\alpha_1(t)\,y(t)\,.
\end{equation*}
From~\eqref{strognsublinia} and $v_1(t)\leq w_1(t)$, we deduce that for each $\delta>0$
\begin{equation*}
v_1'(t)-l_1(t,v_1(t))= \lambda \,h(t,x_m(t,f,x_0))- h(t,\lambda\,x_m(t,f,x_0))<0=w_1'(t)-l_1(t,w_1(t))
\end{equation*}
for $t$ in a subset of $(0,\delta)$ with positive Lebesgue measure. Therefore, condition (ii) of~\cite[Theorem 4]{paper:Walter} holds and we conclude that $v_1(t)<w_1(t)$ for each $t>0$.
\par \smallskip
Next, we take $y'(t)=l_2(t,y(t))=x_1(t,f,\lambda \,x_0)-\alpha_2(t)\,y(t)$. As before, but now from $v_1(t)<w_1(t)$ for each $t>0$ and $v_2(t)\leq w_2(t)$ we obtain
\[v_2'(t)-l_2(t,v_2(t))<0=w_2'(t)-l_2(t,w_2(t))\,,  \]
for each $t>0$, which implies $v_2(t)<w_2(t)$ for each $t>0$. The inequalities for the rest of the components are obtained in a similar way, so that $v(t)<w(t)$ for each $t>0$, that is, inequality~\eqref{defi:strongsubordi} holds, and $f$ is a point of strong sublinearity for~\eqref{semiflowhbioche}, as claimed.
\end{proof}
As a consequence, under assumptions~\ref{B1}--\ref{B5}, we can apply the conclusions of Theorem~\ref{theoremE+E-} to the monotone, sublinear and continuous skew-product semiflow~\eqref{semiflowhbioche},  characterizing the points of strong sublinearity on $E_-$ and $E_+$, as in the previous lemma. Therefore, the existence of a unique continuous equilibrium whose graph coincides with the pullback attractor of the system is shown.

\end{document}